\documentclass[11pt]{amsart}

\usepackage{amsfonts, amsmath, amssymb, amscd, amsthm, array}
\usepackage{graphicx}
\usepackage{multicol}
\usepackage{color}
\usepackage{hyperref,url}

\makeatletter
\def\blfootnote{\xdef\@thefnmark{}\@footnotetext}
\makeatother

\newtheorem{thm}{Theorem}[section]
\newtheorem{cor}[thm]{Corollary}
\newtheorem{lemma}[thm]{Lemma}
\newtheorem{prop}[thm]{Proposition}
\newtheorem{prob}[thm]{Problem}
\newtheorem{conj}[thm]{Conjecture}
\theoremstyle{definition}
\newtheorem{df}[thm]{Definition}
\newtheorem*{question}{Question}
\theoremstyle{remark}
\newtheorem{rem}[thm]{Remark}
\newtheorem{ex}[thm]{Example}

\newcommand{\G}{\mathcal{G}}
\newcommand{\ga}{\Gamma}

\newcommand{\T}{\mathcal{T}}
\renewcommand{\L}{\mathcal{L}}
\newcommand{\M}{\mathcal{M}}
\newcommand{\cN}{\mathcal{N}}

\newcommand{\e}{\varepsilon }
\newcommand{\N}{\mathbb{N}}
\newcommand{\Z}{{\mathbb Z}}

\newcommand{\pc}{{\rm Pc}}
\renewcommand{\d}{{\rm d}}

\newcommand{\X}{\mathcal{AH}}
\newcommand{\gN}{\mathcal{N}}

\newcommand{\MM}{{\mathcal M}^3}
\newcommand{\Nn}{{\rm N}}
\newcommand{\C}{{\rm C}}
\newcommand{\st}{{\rm Stab}}
\newcommand{\pst}{{\rm PStab}}

\newcommand{\h}{\hookrightarrow_h}

\DeclareMathOperator{\link}{link}

\DeclareMathOperator{\esupp}{esupp}
\DeclareMathOperator{\pdim}{pdim}

\newcommand{\gen}[1]{\left\langle#1\right\rangle}

\def\coloneqq{\mathrel{\mathop\mathchar"303A}\mkern-1.2mu=}

\begin{document}

\date{}

\begin{abstract}
We provide new examples of acylindrically hyperbolic groups arising from actions on simplicial trees. In particular, we consider amalgamated products and HNN-extensions,
one-relator groups, automorphism groups of polynomial algebras,  $3$-manifold groups and graph products.
Acylindrical hyperbolicity is then used to obtain some results about the algebraic structure, analytic properties and measure equivalence rigidity of groups from these classes.
\end{abstract}



\title{Acylindrical hyperbolicity of groups acting on trees}
\author{Ashot Minasyan}
\address[Ashot Minasyan]{Mathematical Sciences,
University of Southampton, Highfield, Southampton, SO17 1BJ, United Kingdom.}
\email{aminasyan@gmail.com}

\author{Denis Osin}
\thanks{The research of the first author was partially supported by EPSRC grant EP/H032428/1. The research of the second author was supported by
NSF grant  DMS-1006345 and by the RFBR grant 11-01-00945.}
\address[Denis Osin]{Department of Mathematics, Vanderbilt University, Nashville, TN 37240, USA}
\email{denis.v.osin@vanderbilt.edu}

\maketitle
%\tableofcontents

\section{Introduction}
Recall that an isometric action of a group $G$ on a metric space $(S,\d)$ is {\it acylindrical} if for every $\e>0$ there exist $R,N>0$
such that for every two points $x,y$ with $\d (x,y)\ge R$, there are at most $N$ elements $g\in G$ satisfying
$$
\d(x,gx)\le \e \;\;\; {\rm and}\;\;\; \d(y,gy) \le \e.
$$
The notion of acylindricity goes back to Sela's paper \cite{Sel}, where it was considered for groups acting on trees.
In the context of general metric spaces, this concept is due to Bowditch \cite{Bow}. The following definition was suggested by the second author in \cite{Osi13}.

\begin{df}\label{df:acyl-gp}
A group $G$ is called \emph{acylindrically hyperbolic} if there exists a  generating set $X$ (possibly infinite) of $G$
such that the corresponding Cayley graph $\Gamma (G,X)$ is hyperbolic, non-elementary (i.e., its Gromov boundary $\partial \Gamma (G,X)$ consists of more than  $2$ points),
and the action of $G$ on $\Gamma (G,X)$ is acylindrical.
\end{df}

The above notion generalizes non-elementary word hyperbolic groups. Indeed,
it is easy to see that a group is non-elementary word hyperbolic if and only if it satisfies the above definition for a finite set $X$. For $X$ finite,
the acylindricity condition is vacuous as proper and cobounded actions are always acylindrical. On the other hand, if we allow $X$ to be infinite,
hyperbolicity and non-elementarity of $\Gamma (G,X)$ are rather weak assumptions and acylindricity of the action turns out to be \emph{the}
condition that allows one to derive interesting results.

In fact, acylindrically hyperbolic groups were implicitly studied by Bestvina and Fujiwara \cite{BF}, Bowditch \cite{Bow}, Hamenst\" adt \cite{Ham},
Dahmani, Guirardel, and Osin \cite{DGO}, Sisto \cite{Sis}, and many other authors before they were formally defined in \cite{Osi13}.
However the above-mentioned papers used different definitions stated in terms of hyperbolically embedded subgroups \cite{DGO},
weakly contracting elements \cite{Sis}, or various forms of (weakly) acylindrical actions \cite{BF,Bow,Ham}. Some nontrivial relations
between these definitions were established in \cite{DGO} and \cite{Sis}, and finally the equivalence of all definitions was
proved in \cite{Osi13} (see Section \ref{prelim} for details).

The class of acylindrically hyperbolic groups, denoted by $\X$, encompasses many examples of interest: non-(virtually cyclic) groups
hyperbolic relative to proper subgroups, $Out(F_n)$ for $n>1$, all but finitely many mapping class groups, non-(virtually cyclic) groups acting properly on
proper CAT($0$)-spaces and containing rank one elements, and so forth (see \cite{DGO,Osi13}). On the other hand, $\X$ is restricted enough to
possess a non-trivial theory.  Below we mention just few directions in the study of acylindrically hyperbolic groups; for a more
comprehensive survey we refer the reader to \cite{Osi13}.

\begin{itemize}
\item Every acylindrically hyperbolic group contains non-degenerate hyperbolically embedded subgroups. This allows one to use
methods from the paper \cite{DGO} by Dahmani, Guirardel, and Osin to transfer a significant portion of the theory of relatively hyperbolic groups,
including group theoretic Dehn filling, to the class  $\X$. Despite their generality, the techniques from \cite{DGO} are capable of answering open questions
and producing new results even for well-studied classes of groups such as relatively hyperbolic groups, mapping class groups, and $Out(F_n)$.

    \medskip

\item Acylindrical hyperbolicity shares many common features with ``analytic negative curvature" as defined in \cite{MS} and \cite{T}.
In particular, Hamenst\" adt \cite{Ham} showed that $\X$ is a subclass of the Monod--Shalom class $\mathcal{C}_{reg}$ (see also \cite{HO} and \cite{BBF}
for various generalizations). This opens doors for the Monod--Shalom measure rigidity and orbit rigidity theory  \cite{MS}. Results about
quasi-cocycles on acylindrically hyperbolic groups with coefficients in the left regular representation obtained in \cite{BBF,Ham,HO} also
seem likely to be useful in the study of the structure of group von Neumann algebras via a further generalization of methods from \cite{CS,CSU}.

    \medskip

\item Sisto \cite{Sis} studied random walks on acylindrically hyperbolic groups using the language of weakly contracting elements.
His main result covers many classical theorems about random elements of hyperbolic groups, mapping class groups, etc.

    \medskip

\item Yet another direction is explored by Hull \cite{Hull}, who developed  a version of small cancellation theory for acylindrically hyperbolic
groups and used it to construct groups with certain exotic properties.
\end{itemize}

The abundance of non-trivial results and techniques applicable to acylindrically hyperbolic groups justifies the quest for new examples,
which is the main goal of our paper. We concentrate on examples arising from actions on simplicial trees. In particular, we consider fundamental groups
of finite graphs of groups, one-relator groups, automorphism groups of polynomial algebras, fundamental groups of compact orientable $3$-manifolds, and graph products.
To illustrate usefulness of our main results we derive a number of corollaries about the algebraic structure, analytic properties and
measure equivalence rigidity of groups from these classes. However the main focus of this paper is on new examples rather than on applications.

\smallskip
\noindent {\bf Acknowledgements.}
The authors would like to thank Henry Wilton for helpful discussions of $3$-manifolds. We are also grateful to Jack Button for his valuable comments,
and to Stephane Lamy for pointing out an error in an earlier version of the paper. Finally, we would like to thank the referee for a careful reading of this article.

\section{Main results}\label{sec:results}

\subsection{Fundamental groups of graphs of groups}
We begin with a general theorem about groups acting on trees. Recall that the action of a group $G$ by automorphisms on a simplicial tree
$\T$ is called {\it minimal} if $\T$ contains no proper $G$-invariant subtree. As usual, by $\partial \T$ we denote the Gromov boundary of
$\T$, which can be identified with the set of ends of $\T$; no topology on $\partial \T$ is assumed.

\begin{thm}\label{main-1}
Let $G$ be a group acting minimally on a simplicial tree $\T$. Suppose that $G$ does not
fix any point of $\partial \T$ and there exist vertices $u,v$ of $\T$ such that the pointwise stabilizer $\pst_G\{ u,v\}$ is finite.
Then $G$ is either virtually cyclic or acylindrically hyperbolic.
\end{thm}

A slightly different criterion, which does not make use of the boundary is formulated in Corollary \ref{prop:WPD-crit} (see Subsection \ref{sec:trees} below).

If $G$ is the fundamental group of a graph of groups $(\mathcal G,\ga)$, one can apply Theorem \ref{main-1} to the action of $G$ on the associated Bass-Serre tree.
In this case the minimality of the action and the absence of fixed points on $\partial \T$ can be recognized from the
local structure of $(\mathcal G,\ga)$ -- see Section \ref{Sec-gg}. We mention here two particular cases and refer to Theorem \ref{thm:reduced-crit}  for a more general result.
Below we say that a subgroup $C$ of a group $G$ is \emph{weakly malnormal} if there exists $g\in G$ such that $|C^g \cap C|<\infty$.

\begin{cor}\label{cor:amalg-intr}
Let $G$ split as a free product of groups $A$ and $B$ with an amalgamated subgroup $C$. Suppose $A\ne C\ne B$ and $C$ is weakly malnormal in $G$.
Then $G$ is either virtually cyclic or acylindrically hyperbolic.
\end{cor}

Note that the virtually cyclic case cannot be excluded from Corollary \ref{cor:amalg-intr}. Indeed it is realized if $C$ is finite and has index $2$ in both factors.
However these are the only exceptions: it is easy to show that under the assumptions of the corollary all other amalgams contain
non-abelian free subgroups, %(cf. \cite[Thm. 6.1]{Bass76})
and so they are are acylindrically hyperbolic.

\begin{cor}\label{cor:HNN-intr}
Let $G$ be an HNN-extension of a group $A$ with associated subgroups $C$ and $D$. Suppose that $C\ne A\ne D$ and $C$ is weakly malnormal in $G$. Then $G$ is acylindrically hyperbolic.
\end{cor}

\begin{rem} \label{rem:HNN} The assumption of Corollary \ref{cor:HNN-intr} saying that $C$ is weakly malnormal in $G$ is equivalent to the existence of an element $g \in G$ such that $|C^g \cap D|<\infty$,
because $C$ and $D$ are conjugate in $G$.
\end{rem}

The weak malnormality condition is essential in both corollaries. This can be illustrated in many ways, but the most convincing examples are the following.
\begin{ex} \label{ex:non-ah-1}
\begin{enumerate}
\item[(a)] There are simple groups splitting as $A\ast _C B$, where $A$, $B$ are finitely generated free groups and $C$ can be of finite or infinite index in both
$A$ and $B$ \cite{BM, C}. Note that simple groups are never acylindrically hyperbolic (see Theorem~\ref{thm:elem-prop-ah}).
\item[(b)] The Baumslag-Solitar groups $BS(m,n)=\langle a,t \mid t^{-1}a^mt=a^n \rangle $  are not acylindrically hyperbolic unless $m=0$ or $n=0$ \cite{Osi13}.
\end{enumerate}
\end{ex}
Furthermore, the assumption that $C\ne A\ne D$ in Corollary \ref{cor:HNN-intr} is also necessary, see Examples \ref{ex:HNNB} and \ref{ex:bg}.

It is worth noting that our  Corollaries \ref{cor:amalg-intr} and \ref{cor:HNN-intr} can be used to recover some results proved by
Schupp \cite{S} and Sacerdote and Schupp \cite{SS} by using small cancellation theory for amalgamated products and HNN-extensions.
In fact, our approach is slightly more general; for details we refer to Section \ref{Sec-gg}.

The next result is inspired by \cite{SS}. The proof makes use of Corollary \ref{cor:HNN-intr} and a technical result from \cite{SS}.

\begin{cor}\label{cor:one-rel}
Let $G$ be a one-relator group with at least $3$ generators. Then $G$ is acylindrically hyperbolic.
\end{cor}

Corollary \ref{cor:one-rel} is related to a conjecture of P. Neumann \cite{PN}, claiming that if $G$ is a one-relator group then either $G$ is cyclic or $G \cong BS(1,n)$ is a solvable Baumslag-Solitar group, or
$G$ is SQ-universal. Recall that a group $G$ is said to be \emph{SQ-universal} if every countable group embeds in a quotient of $G$. The first non-trivial example of an
SQ-universal group was provided by Higman, Neumann and Neumann \cite{HNN}, who proved that the free group of rank $2$ is SQ-universal or,
equivalently, every countable group embeds in a group generated by $2$ elements. It is straightforward to see that any SQ-universal group contains
non-abelian free subgroups. Furthermore, since the set of all finitely generated groups is uncountable and  every countable group contains countably many finitely
generated subgroups,  every countable SQ-universal group has uncountably many non-isomorphic quotients. Thus the property of being
SQ-universal may be thought of as an indication of ``algebraic largeness" of a group. For a survey of SQ-universality we refer to \cite{AMO,Schupp}.

Sacerdote and Schupp \cite{SS} established the SQ-universality of one-relator groups with at least $3$ generators, thus confirming Neumann's conjecture in this case.
Corollary \ref{cor:one-rel} can be regarded as a strengthening of this result
(indeed, by a theorem of Dahmani, Guirardel and Osin \cite{DGO}, every acylindrically hyperbolic group is SQ-universal).
The case of one-relator groups with two generators is more complicated, as one can see by looking at the Baumslag-Solitar groups.
The full proof of Neumann's conjecture was announced by Sacerdote \cite{Sac}, but it was never published, and it seems that the conjecture is still open.
In Section \ref{sec:1-rel}, we make a step towards the resolution of this conjecture by showing that
every one-relator group with a sufficiently complicated Magnus-Moldavanskii hierarchy is
acylindrically hyperbolic (see Proposition~\ref{cor:1-rel-2-gen}). More precisely, all non-acylindrically hyperbolic one-relator groups belong to the class of  HNN-extensions of
(finitely generated free)-by-cyclic groups. Unfortunately, our techniques do not apply to ascending HNN-extensions (cf. Examples \ref{ex:HNNB}, \ref{ex:bg}), therefore for groups
with Magnus-Moldavanskii hierarchy of low complexity we only have
partial results.

Another example having the structure described in Corollary \ref{cor:amalg-intr} is ${\rm Aut}\,k[x,y]$, the group of automorphisms of the polynomial algebra $k[x,y]$
for any field $k$. This group is sometimes called the \emph{integral 2-dimensional Cremona groups} over $k$ and is also isomorphic to
the group of automorphisms of the free associative algebra $k\langle x,y\rangle$ \cite{Cze,ML}; if ${\rm char}\, k=0$ it is also
isomorphic to the group of automorphisms of the free $2$-generated Poisson algebra \cite{MLTU}. The algebraic structure of these
groups has received a lot of attention but is still far from being well understood, see \cite{BNS,FL,Wri} and references therein. We prove the following.

\begin{cor}\label{GA2}
For any field $k$, the group ${\rm Aut}\,k[x,y]$ is acylindrically hyperbolic.
\end{cor}

For motivation and a survey of some related results we refer to Section \ref{Sec-gg}. Note that ${\rm Aut}\,k[x]$ is solvable and hence it is not
acylindrically hyperbolic. We do not know if the analogue of Corollary \ref{GA2} holds true for the group ${\rm Aut}\,k[x_1,\ldots, x_n]$ when $n\ge 3 $,
but we would rather expect some higher rank phenomena to prevent it from being acylindrically hyperbolic.

Other examples of acylindrically hyperbolic groups obtained in Section \ref{Sec-gg} by using Corollaries \ref{cor:amalg-intr} and \ref{cor:HNN-intr}
include the Higman group $$H=\langle a,b,c,d \mid a^b=a^2, \, b^c=b^2,\, c^d=c^2,\, d^a=d^2\rangle $$ and
amalgamated products (respectively, HNN-extensions) of the form $A\ast _BC$ (respectively, $A\ast_C$), where $A$ is a
hyperbolic group and $C$ is a quasi-convex subgroup of infinite index in $A$. The latter examples should be compared to
simple groups mentioned in Example \ref{ex:non-ah-1}. For details we refer to Section \ref{sec:misc}

\subsection{$3$-manifold groups}
By default, all manifolds considered in this paper are assumed to be connected. Our next goal is to show that most orientable $3$-manifold groups are
acylindrically hyperbolic.  More precisely, let $\MM $ denote the class of all \emph{subgroups} of fundamental groups of compact orientable $3$-manifolds.
Note that by Scott's theorem \cite{Sco73}, every finitely generated group $G\in \MM$ is itself the fundamental group of a compact orientable $3$-manifold.
However the groups we consider are not necessarily finitely generated.

\begin{thm}\label{main-3-dim}
Let $G\in \MM$. Then exactly one of the following three conditions holds.
\begin{enumerate}
\item[(I)] $G$ is acylindrically hyperbolic.

\item[(II)] $G$ contains an infinite cyclic normal subgroup $Z$ and $G/Z$ is acylindrically hyperbolic.

\item[(III)] $G$ is virtually polycyclic.
\end{enumerate}
\end{thm}

If $G$ is itself a fundamental group of a compact irreducible orientable $3$-manifold, we can give a more geometric description of the non-acylindrically hyperbolic cases.

\begin{cor}\label{cor:geom}
Let $M$ be a compact orientable irreducible $3$-manifold
such that $\pi_1(M)$ is not acylindrically hyperbolic. Then either $\pi_1(M)$ is virtually polycyclic or $M$ is Seifert fibered.
\end{cor}

Theorem \ref{main-3-dim} allows to bring many results known for acylindrically hyperbolic groups to the world of $3$-manifold groups.
We discuss some examples below. Let us say that a group $G\in \MM$ has \emph{type} (I), (II), or (III) if the corresponding condition from Theorem \ref{main-3-dim} holds for $G$.

\begin{cor} \label{3-dim-cor-1}
A group $G\in \MM$ is inner amenable iff it is of type (II) or (III).
\end{cor}

Recall that a group $G$ is \emph{inner amenable} if there exists a finitely additive conjugacy invariant probability measure on $G\setminus \{ 1\}$. Inner amenability is closely related to the Murray--von Neumann property $\Gamma $ for operator algebras. In particular, if $G$ is not inner amenable, the von Neumann algebra $W^\ast (G)$ of $G$ does not have property $\Gamma $ \cite{Efr}. For further details and motivation we refer to \cite{BeHa} and Section \ref{sec:3m-app}.

\begin{cor}\label{3-dim-cor-2}
Type is a measure equivalence invariant in $\MM$. That is, if $G,H\in \MM$ and $G$ is measure equivalent to $H$, then $G$ and $H$ have the same type.
\end{cor}

This result can be seen as the first step towards measure equivalence classification of $3$-manifold groups. We note that a classification of fundamental groups of compact $3$-manifolds up to quasi-isometry was done
by Behrstock and  Neumann \cite{BN,BN1}, while no non-trivial results about measure equivalence seem to be known.

Finally we mention that Theorem \ref{main-3-dim} is used in \cite{A-M-S} to show the following:
\emph{the outer automorphism group ${\rm Out}(\pi_1(M))$ is residually finite for any compact $3$-manifold $M$}.  The proof is based on the approach suggested in \cite{MO}
and essentially uses acylindrical hyperbolicity.

\subsection{Graph products}\label{subsec:gp}
Let $\Gamma $ be a graph without loops and multiple edges with vertex set $V(\Gamma)=V$. Let also  $\{G_v\}_{v\in V}$  be a family of groups, called \emph{vertex groups},
indexed by vertices of $\Gamma$. The \emph{graph product} of $\{G_v\}_{v\in V}$ with respect to $\Gamma $, denoted  $\Gamma \{G_v\}_{v\in V}$,
is the quotient group of the free product $\ast_{v\in V} G_v$ by the relations $[g,h]=1$ for all $g\in G_u$, $h\in G_v$ whenever $u$ and $v$ are
adjacent in $\Gamma$. Graph products generalize free and direct products of groups. Other basic examples include right angled Artin and
Coxeter groups. The study of graph products and their subgroups has gained additional importance in view of the recent breakthrough results of
Agol, Haglund, Wise, and their co-authors, which claim that many groups can be virtually embedded into right angled Artin groups (see \cite{Agol,HW,Wise-qch} and references therein).

For any subset $U \subseteq V$, the subgroup $\langle G_u \mid u \in U\rangle$ of the graph product $G=\Gamma \{G_v\}_{v\in V}$ is said to be a \emph{full subgroup}; it is
naturally isomorphic to to the graph product $\Gamma_U \{G_u\}_{u\in U}$, where $\Gamma_U$ is the full subgraph of $\Gamma$ spanned on the vertices from $U$.
Any conjugate of a full subgroup is called \emph{parabolic}. We will say that the graph $\ga$ is \emph{irreducible} if its complement graph $\ga^c$ is connected; this equivalent to
saying that the graph product $\Gamma \{G_v\}_{v\in V}$ does not naturally split as a direct product of two proper full subgroups.

Graph products naturally decompose as amalgamated products in many ways (see Subsection~\ref{sec:wpd_in_gp}). In Section \ref{sec:GP} we study subgroups of graph products using the
actions on the associated Bass-Serre trees together with the theory of parabolic subgroups developed in \cite{A-M-Tits}. In particular, we prove the following (see Section 6.3).

\begin{thm} \label{cor:gp-hyp} Let $G=\ga\{G_v\}_{v\in V}$ be the graph product of non-trivial groups with respect to some finite irreducible graph $\Gamma$ with at least two vertices.
Suppose that $H \leqslant G$ is a subgroup that is not contained in a proper parabolic subgroup of $G$. If $H$ is not virtually cyclic then $H \in \X$.
\end{thm}

The corollary below covers the particular case when $H=G$.

\begin{cor}\label{thm:gpacyl}
Let $G=\Gamma \{G_v\}_{v\in V}$ be the graph product of non-trivial groups with respect to some finite irreducible
graph $\Gamma$ with at least two vertices. Then $G$ is either virtually cyclic or acylindrically hyperbolic.
\end{cor}

Graph products are simplest examples of groups having ``relative non-positive curvature"  with respect to the vertex subgroups. From this point of view, Corollary \ref{thm:gpacyl} can be considered as another instance of a more general phenomenon:
``irreducibility" of ``nice" non-positively curved groups (or spaces) often implies the existence of ``hyperbolic directions"
in an appropriate sense; in turn, the latter condition implies acylindrical hyperbolicity. In the context of groups acting on $CAT(0)$ spaces,
a more formal version of this claim  is sometimes called the \emph{Rank Rigidity Conjecture} \cite{BB,CS} and is closely related to the
classical work of Ballman, Brin, Burns, and Spatzier \cite{Bal,Bal95,BS87} as well as more recent results of Caprace and Sageev \cite{CS}.
For a more detailed discussion we refer to Section \ref{sec:subgr-class}.

Theorem \ref{cor:gp-hyp} is quite useful for studying properties of subgroups of graph products. For example, it is used  in \cite{A-M-S} to prove that $Out(G)$ is residually finite for any virtually compact special
(in the sense of Haglund and Wise \cite{HW}) group $G$.
Theorem \ref{cor:gp-hyp} also implies that every subgroup $H$ of a graph product $G$ is either acylindrically hyperbolic, or  ``reducible'', or ``elementary'' in an appropriate sense
(see Theorem \ref{thm:subgroups}). A particular application to  right angled Artin groups gives rise to the
following algebraic alternative:

\begin{cor}\label{cor:RAAG_in_X} Let $H$ be a non-cyclic subgroup of a finitely generated right angled Artin group. Then exactly one of the following holds:
\begin{itemize}
  \item $H$ is acylindrically hyperbolic;
  \item $H$ contains two non-trivial normal subgroups $N_1,N_2 \lhd H$ such that $N_1 \cap N_2=\{1\}$.
\end{itemize}
\end{cor}

We also mention one application of graph products to groups with hyperbolically embedded subgroups. Recall that if a group $G$ is hyperbolic relative to a subgroup $H$,
then many ``nice'' properties of $H$ (e.g., solvability of algorithmic problems, finiteness of the asymptotic dimension, various analytic properties, etc.)
are inherited by $G$ \cite{Farb,Osi06,O-asdim,DS}. As shown in \cite{DGO},  many properties of acylindrically hyperbolic groups resemble those of
relatively hyperbolic groups, and hyperbolically embedded subgroups (see Subsection \ref{sec:eqdef} for the definition)
serve as analogues of peripheral subgroups in the relatively hyperbolic case.
Thus one may wonder if the structure of an acylindrically hyperbolic group is determined, to some extent, by the structure of its hyperbolically embedded subgroups.
The following result provides a strong negative answer; it is a simplified version of Theorem \ref{thm:emb}.

\begin{thm}
Any finitely generated group can be embedded into a finitely generated acylindrically hyperbolic group $G$ such that every proper
hyperbolically embedded subgroup of $G$ is finitely generated and virtually free.
\end{thm}

\subsection{Geometric vs. analytic negative curvature}
Recall that a countable group $G$ belongs to the Monod--Shalom class $\mathcal C_{reg}$ if $H_b^2(G, \ell^2(G))\ne 0$ \cite{MS}.
An a priori different class $\mathcal D_{reg}$ was defined by Thom \cite{T}: a countable group $G$ is in $\mathcal D_{reg}$ if $G$ is
non-amenable and there exists an unbounded quasi-cocycle $G\to \ell^2(G)$.  Groups from $\mathcal C_{reg}$ and $\mathcal D_{reg}$
exhibit behavior typical to non-elementary hyperbolic groups. Moreover, most known examples from these classes come from the world
of groups acting on hyperbolic spaces. Motivated by these observations, Monod \cite{Mon} and Thom  \cite{T} suggested to consider
$\mathcal C_{reg}$ and $\mathcal D_{reg}$ as analytic analogues of the class of ``negatively curved" groups. It is worth noting that
no non-trivial relation between $\mathcal C_{reg}$ and  $\mathcal D_{reg}$ is known, although they are likely to coincide (see Remark 2.7 in \cite{T})

Monod \cite[Problem J]{Mon} asked whether $\mathcal C_{reg}$ can be characterized by a geometric condition. In this paper we discuss the relation between $\mathcal C_{reg}$, $\mathcal D_{reg}$, and $\X $.
First we recall that $\X\subseteq \mathcal C_{reg}\cap \mathcal D_{reg}$. This was originally proved by Hamenst\"adt \cite{Ham} in slightly different terms and also follows from the results of
Hull-Osin \cite{HO} and Bestvina--Bromberg--Fujiwara \cite{BBF}. Moreover, acylindrically hyperbolic groups constitute the majority of examples from $\mathcal C_{reg}$ or $\mathcal D_{reg}$. Indeed the only examples of non-acylindrically hyperbolic groups
from $\mathcal C_{reg}\cup \mathcal D_{reg}$ known up to now are those constructed by Osin \cite{Osi09} and L\"uck and Osin \cite{LO}.
These are finitely generated infinite torsion groups with positive first $\ell^2$-betti numbers. Note that $\beta_1^{(2)}(G)>0$ implies $G\in \mathcal D_{reg}$ while torsion groups are never acylindrically hyperbolic (see Example \ref{ex:non-ah} (b)).

Groups from \cite{Osi09} and \cite{LO} are not finitely presented. On the other hand, there is an evidence that for finitely presented groups, existence of non-trivial quasi-cocycles may be used to construct non-elementary actions on hyperbolic spaces \cite{Man}. This motivates the following question:

\begin{question}[Geometric vs. Analytic Negative Curvature]
Are the conditions $G\in {\mathcal C}_{reg}$, $G\in {\mathcal D}_{reg}$, and $G\in \mathcal{AH}$ equivalent for a finitely presented group $G$?
\end{question}

We do not really believe in a positive answer to this question in full generality.
However, in Section \ref{sec:nc} we obtain the following as immediate applications of other results from the paper.

\begin{cor}\label{cor:GANC} The conditions $G\in {\mathcal C}_{reg}$, $G\in {\mathcal D}_{reg}$, and $G\in \mathcal{AH}$ are equivalent for any group $G$
from the following classes:
\begin{enumerate}
\item[(a)] Subgroups of fundamental groups of compact orientable $3$-manifolds.
\item[(b)] Subgroups of graph products of amenable groups.
In particular, this class includes subgroups of right angled Artin groups.
\end{enumerate}
\end{cor}

Recall that two groups $G_1$ and $G_2$ are {commensurable up to a finite kernel} if there are finite index subgroups $H_i\leqslant G_i$ and finite normal subgroups
$N_i\lhd H_i$, $i=1,2$, such that $H_1/N_1\cong H_2/N_2$. For definitions, details, and related questions we refer to Section \ref{sec:nc}
and \cite{HW}.

\section{Preliminaries}\label{prelim}
In this section we collect the main definitions and results about acylindrically hyperbolic groups relevant to our paper. For more details we refer to \cite{DGO,Osi13}.

\subsection{Equivalent definitions of acylindrically hyperbolic groups}\label{sec:eqdef}

Given a group $G$ acting on a hyperbolic space $S$, an element $g\in G$ is called \emph{elliptic} if some (equivalently, any) orbit of
$g$ is bounded, and \emph{loxodromic} (or \emph{hyperbolic}) if the map $\mathbb Z\to S$ defined by $n\mapsto g^ns$ is a quasi-isometry for some (equivalently, any) $s\in S$.

\begin{rem}
In papers devoted to groups acting on general hyperbolic spaces, the terms ``loxodromic" and ``hyperbolic" are used as synonyms.
Recent papers on relatively hyperbolic and acylindrically hyperbolic groups (including papers by the second author) tend
to use the term ``loxodromic" more often. In this paper we switch to ``hyperbolic" since we mostly deal with actions on trees for which this term is well established.
\end{rem}

Every hyperbolic element $g\in G$ has exactly $2$ limit points $g^{\pm\infty}$ on the Gromov's boundary $\partial S$.
Hyperbolic elements $g,h\in G$ are called \emph{independent} if the sets $\{ g^{\pm \infty}\}$ and $\{ h^{\pm \infty}\}$ are disjoint.
The next theorem classifies acylindrical group actions on hyperbolic spaces.
\begin{thm}[{\cite[Theorem 1.1]{Osi13}}]\label{class}
Let $G$ be a group acting acylindrically on a hyperbolic space. Then $G$ satisfies exactly one of the following three conditions.
\begin{enumerate}
\item[(a)] $G$ has bounded orbits.
\item[(b)] $G$ is virtually cyclic and contains a hyperbolic element.
\item[(c)] $G$ contains infinitely many independent hyperbolic elements.
\end{enumerate}
\end{thm}

Compared to the general classification of groups acting on hyperbolic spaces, Theorem~\ref{class} rules out parabolic and quasi-parabolic actions in Gromov's terminology \cite{Gro}.
For details we refer to \cite{Osi13}.

Recall that the action of a group $G$ on a hyperbolic space $S$ is called {\it elementary} if the limit set of $G$ in the Gromov boundary $\partial S$
consists of at most $2$ points and {\it non-elementary} otherwise. According to Theorem \ref{class}, an acylindrical action of $G$ is non-elementary if and only if the orbits are unbounded and $G$ is not virtually cyclic.

The next theorem provides equivalent characterizations of acylindrically hyperbolic groups. In practice, (AH$_3$) is the most useful condition for proving that a
certain group is acylindrically hyperbolic, while (AH$_1$), (AH$_2$), and (AH$_4$) are more useful for proving theorems about acylindrically hyperbolic groups.

\begin{thm}[{\cite[Theorem 1.2]{Osi13}}]\label{acyl}
For any group $G$, the following conditions are equivalent.
\begin{enumerate}
\item[(AH$_1$)] The group $G$ is acylindrically hyperbolic in the sense of Definition \ref{df:acyl-gp}.
\item[(AH$_2$)] $G$ admits a non-elementary acylindrical action on a hyperbolic space.
\item[(AH$_3$)] $G$ is not virtually cyclic and admits an action on a hyperbolic space such that at least one element of $G$ is hyperbolic and satisfies the WPD condition.
\item[(AH$_4$)] $G$ contains a proper infinite hyperbolically embedded subgroup.
\end{enumerate}
\end{thm}

Note that every group has an obvious acylindrical action on a hyperbolic space, namely the trivial action on a point.
Thus considering elementary acylindrically hyperbolic groups does not make much sense. For this reason we include non-elementarity in the definition.
The terminology used in conditions (AH$_3$) and (AH$_4$) will also be used in our paper, so we explain it below.

Given a group $G$ acting on a metric space $(S,\d)$, a subset $A\subseteq S$, and $\e\ge 0$, we define the \emph{pointwise $\e$-stabilizer} of $A$ in $G$
as the set of all $g\in G$ that move every point of $A$ by at most $\e$. That is,
$$
\pst_G^\e(A)=\{ g\in H\mid \d (ga, a)\le \e\; \forall \, a\in A\} .
$$
Thus the pointwise $0$-stabilizer is the same as the usual pointwise stabilizer, which will be denoted $\pst_G(A)$. Note that, in general, pointwise $\e$-stabilizers are not subgroups.

\begin{rem}\label{rem:equiv_def-acyl}
In these terms, the action of $G$ on $S$ is acylindrical if and only if for every $\e>0$ there exists $R>0$ such that for all $x,y\in S$ with $\d (x,y)\ge R$,
the sizes of pointwise $\e$-stabilizers of $\{ x,y\}$ are uniformly bounded by a constant $N$ which only depends on $\e$.
\end{rem}

\begin{df} \label{def:wpd} An element $h$ of a group $G$ acting isometrically on a metric space $S$ satisfies the {\it weak proper discontinuity}
condition (or $h$ is a {\it WPD element}) if for some $s\in S$ (or, equivalently, for all $s \in S$) and every $\e \ge 0$, exists $M\in \mathbb N$ such that
\begin{equation*}%\label{eq: wpd}
|\pst^\e_G(\{ s, h^Ms\} )|<\infty .
\end{equation*}
This definition is essentially due to Bestvina and Fujiwara \cite{BF}.
\end{df}

The notion of a hyperbolically embedded subgroup was introduced in \cite{DGO}; it generalizes the notion of a peripheral
subgroup of a relatively hyperbolic group. More precisely, let $G$ be a group, $H\leqslant G$, $X\subseteq G$.
We assume that $G=\langle X\cup H\rangle $ and denote by $\Gamma(G, X\sqcup H)$ the Cayley graph of $G$ with
respect to the generating set $X\sqcup H$ (even though $X$ and $H$ might intersect as subsets of $G$, for the purposes of constructing $\Gamma(G, X\sqcup H)$
we consider them to be disjoint) and by $\Gamma _H$ the Cayley graph of $H$ with respect to the generating set $H$.
Clearly $\Gamma _H$ is a complete subgraph of $\Gamma (G, X\sqcup H)$.

Given two elements $h_1,h_2\in H$, we define $\widehat d(h_1,h_2)$ to be the length of a shortest path $p$ in $\Gamma (G, X\sqcup H) $ that connects $h_1$ to $h_2$
and does not contain edges of $\Gamma _H$. If no such path exists we set $\widehat d(h_1, h_2)=\infty $. Clearly $\widehat d\colon H\times H\to [0, \infty]$
is an extended metric on $H$.
We say that $H$ is {\it hyperbolically embedded in $G$ with respect to} $X\subseteq G$ (and write $H\hookrightarrow _h (G,X) $) if the following conditions hold:
\begin{enumerate}
\item[(a)] $G=\langle X\cup H\rangle$ and $\Gamma (G, X\sqcup H)$ is hyperbolic.
\item[(b)] $(H,\widehat{d})$ is a locally finite space, i.e., every ball (of finite radius) is finite.
\end{enumerate}
We also say that $H$ is \textit{hyperbolically embedded in $G$} (and write $H\hookrightarrow _h G$) if $H\hookrightarrow _h (G,X) $ for some $X\subseteq G$.

Let us consider three basic examples.

\begin{ex}\label{hes}
\begin{enumerate}
\item For any group $G$, we have $G\h G$.  Indeed take $X=\emptyset $. Then the Cayley graph $\Gamma(G, X\sqcup H)$ has diameter $1$ and $\widehat d(h_1, h_2)=\infty $ whenever $h_1\ne h_2$. Further, if $H$ is a finite subgroup of a group $G$, then $H\h G$. Indeed $H\h (G,X)$ for $X=G$. These cases are referred to as {\it degenerate}.

\item Let $G=H\times \mathbb Z$, $X=\{ x\} $, where $x$ is a generator of $\mathbb Z$. Then $\Gamma (G, X\sqcup H)$ is quasi-isometric to a line and hence it is hyperbolic. However $\widehat d(h_1, h_2)\le 3$ for every $h_1, h_2\in H$. If $H$ is infinite, then $H\not\h (G,X)$, and moreover $H\not\h G$.

\item Let $G=H\ast \mathbb Z$, $X=\{ x\} $, where $x$ is a generator of $\mathbb Z$. In this case $\Gamma (G, X\sqcup H)$ is quasi-isometric to a tree and $\widehat d(h_1, h_2)=\infty $ unless $h_1=h_2$. Thus $H\h (G,X)$.
\end{enumerate}
\end{ex}

It is known that a group $G$ is hyperbolic relative to a subgroup $H$ if and only if $H\hookrightarrow _h (G,X) $ for some finite set $X$ \cite{DGO}.

\subsection{Properties of acylindrically hyperbolic groups}\label{sec:ah-prop}
We begin with some elementary algebraic properties of acylindrically hyperbolic groups. The theorem below summarizes some results
from \cite[Section 7]{Osi13} and \cite[Theorems 2.24, 2.27, 2.33]{DGO}. Recall that a subgroup $H\leqslant G$ is \emph{$s$-normal} in $G$
if for every $g\in G$, we have $|H\cap H^g|=\infty$. Clearly every infinite normal subgroup is $s$-normal. Recall also that a group $G$ is SQ-universal
if every countable group embeds in a quotient  of $G$.

\begin{thm}[Dahmani--Guirardel-Osin \cite{DGO}, Osin \cite{Osi13}]\label{thm:elem-prop-ah}
For any acylindrically hyperbolic group $G$ the following conditions hold.
\begin{enumerate}
\item[(a)] The amenable radical of $G$ is finite.

\item[(b)] If $G=G_1\times G_2$, then $|G_i|<\infty $ for $i=1$ or $i=2$.

\item[(c)] Every $s$-normal subgroup of $G$ is acylindrically hyperbolic.

\item[(d)] If $G=G_1\ldots G_n$ for some subgroups $G_1, \ldots, G_n$ of $G$. Then $G_i$ is acylindrically hyperbolic for at least one $i$. In particular, $G$ is not boundedly generated.
\item[(e)] $G$ is SQ-universal. In particular, $G$ contains non-abelian free subgroups.
\item[(f)] $G$ contains uncountably many normal subgroups.
\end{enumerate}
\end{thm}

It is useful to keep this theorem in mind when searching for examples of  acylindrically hyperbolic groups as it allows one to show that
certain groups are \emph{not} acylindrically hyperbolic. %Here are some examples.

\begin{ex}\label{ex:non-ah}
\begin{enumerate} \item[]
\item[(a)] No amenable group is acylindrically hyperbolic.

\item[(b)] No group without non-abelian free subgroups is acylindrically hyperbolic. In particular, no torsion group or a group satisfying a non-trivial identity is acylindrically hyperbolic.

\item[(c)]The Baumslag-Solitar groups $$BS(m,n)=\langle a,t\mid t^{-1}a^mt=a^n\rangle $$ are not acylindrically hyperbolic unless $m=0$ or $n=0$, because
the cyclic subgroup $\langle a \rangle$ is $s$-normal in $BS(m,n)$ in this case.
\item[(d)]
$\mathrm{SL}(n,\mathbb Z)$ is not acylindrically hyperbolic for $n\ge 3$, since it is boundedly generated. Another argument is based on the Margulis Theorem,
which states that for $n\ge 3$ every normal subgroup of $\mathrm{SL}(n, \mathbb Z)$ is either finite or of finite index and hence $\mathrm{SL}(n, \mathbb Z)$
has only countably many normal subgroups. For a generalization, see \cite[Example 7.5]{Osi13}.
\end{enumerate}
\end{ex}

%Recall that two groups $G_1$ and $G_2$ are said to be \emph{commensurable up to finite kernels},
%$G_1 \fk G_2$, if there are finite index subgroups $K_i \leqslant G_i$, $i=1,2$,
%a group $H$ and epimorphisms $\xi_i:K_i \to H$ such that $\ker(\xi_i)$ is finite for $i=1,2$. It is not difficult to check that $\fk$ is the
%smallest equivalence relation whose classes are closed under taking subgroups of finite index and quotients modulo finite normal subgroups.
%We begin with two auxiliary results about acylindrically hyperbolic groups.

\begin{lemma} \label{lem: com} Let $G$ be an acylindrically hyperbolic group. Suppose that $H$ is a finite index subgroup of $G$, or a quotient of $G$ modulo a finite normal subgroup, or an extension of $G$ with finite kernel. Then $H$ is also acylindrically hyperbolic.
\end{lemma}

\begin{proof}
Firstly, if $G\in \X$ and $H$ is a finite index subgroup of $G$,
then $H$ is $s$-normal in $G$ and thus $H\in \X$ by Theorem \ref{thm:elem-prop-ah}.

Further, let $G$ be a group, $N\lhd G$ a finite normal subgroup and let $\phi:G \to G/N$ be the natural epimorphism. Then for any generating set $X$ of $G$, $\phi(X)$ generates $G/N$
and the natural map $\ga(G,X) \to \ga(G/N, \phi(X))$ is a $G$-equivariant quasi-isometry.
Thus if $\ga(G,X)$ is non-elementary hyperbolic then so is $\ga(G/N,\phi(X))$.
Also observe that the action of $G$ on $\ga(G/N,\phi(X))$ factors through the canonical action of $G/N$.
Hence if the former is acylindrical then so is the latter. Thus if $G \in \X$ then $G/N \in \X$.

Conversely, if $Y$ is some generating set of $G/N$ then $X=\phi^{-1}(Y)$ is a generating set of $G$ and the Cayley graphs $\ga(G,X)$ and $\ga(G/N,Y)$  are $G$-equivariantly quasi-isometric.
Since $|N|<\infty$, if the action of $G/N$ on $\ga(G/N,Y)$ is acylindrical, then so is the natural action of $G$ on $\ga(G/N,Y)$. This allows to argue as above to conclude that
$G/N \in \X$ implies $G \in \X$.
\end{proof}

The next two results will be used in Section 7 of our paper. The first one was proved by Dahmani, Guirardel and the second author in \cite[Theorems 2.24 and 2.35]{DGO}.

\begin{thm}[{\cite[Theorem 2.32]{DGO}}]\label{thm:prop-ah-2}
For any acylindrically hyperbolic group $G$, there exists a maximal normal finite subgroup $K(G)\lhd G$ and the following conditions are equivalent.
\begin{enumerate}
\item[(a)] $K(G)=\{ 1\}$.
\item[(b)] $G$ has infinite conjugacy classes (equivalently, the von Neumann algebra $W^\ast (G)$ of $G$ is a $II_1$ factor).
\item[(c)] $G$ is not inner amenable. In particular, $W^\ast (G)$ does not have property $\Gamma $ of Murray and von Neumann.
\item[(d)] The reduced $C^\ast $-algebra of $G$ is simple with unique trace.
\end{enumerate}
\end{thm}

The next theorem was first proved in \cite{Ham} under a certain assumption equivalent to acylindrical hyperbolicity (see \cite{Osi13}); it also follows from the main result of \cite{HO}, where the language of hyperbolically embedded subgroups was used. For details about quasi-cocycles and bounded cohomology we refer to \cite{Ham,HO,Mon,T}.

\begin{thm}[Hamenst\"adt \cite{Ham}]\label{thm:prop-ah-3}
$\X \subseteq \mathcal C_{reg}\cap \mathcal D_{reg}$.
\end{thm}

Finally we mention a theorem from \cite{HO13}. Recall that the conjugacy growth function of a finitely generated group for every $n \in \N$ measures the number of conjugacy classes intersecting
the ball of radius $n$ centered at $1$ in the word metric. For details we refer to \cite{HO13}.

\begin{thm}[Hull-Osin \cite{HO13}]\label{thm:prop-ah-4}
Every acylindrically hyperbolic group has exponential conjugacy growth.
\end{thm}

\section{Fundamental groups of graphs of groups}

\subsection{Hyperbolic WPD elements in groups acting on trees}\label{sec:trees}

Given an oriented edge $e$ of a graph $\ga$, $e_-$ and $e_+$ will denote the initial vertex and the terminal vertex of $e$ respectively, and
$\bar e$ will denote the inverse edge to $e$ (i.e., $e$ equipped with the opposite orientation).

In this paper, a tree means a simplicial tree, unless specified otherwise. Let $\mathcal T$ be a tree.
The natural metric on $\T$, induced by identification of edges with $[0,1]$, is denoted by $\d_\T$.
We will think of $\T $ as a simplicial tree and a metric space simultaneously; in particular we will talk about vertices and edges as well as
points of $\T$. Given two points $x,y\in \T$, $[x,y]$ will denote the unique geodesic segment connecting $x$ and $y$.

Throughout this subsection let $G$ denote a group acting on $\T$ by isometries. It is well known that any element $g \in G$ either fixes a point of $\mathcal T$
(which is either a vertex or the midpoint of an edge) or there exists a unique minimal $\gen{g}$-subtree of $\mathcal T$, which is a bi-infinite
geodesic path, called the \textit{axis} of $g$ and denoted by $axis (g)$, on which $g$ acts by translation \cite{Tits}. In the former case $g$ is said to
be \textit{elliptic}, and in the latter case it is said to be \textit{hyperbolic}. For a hyperbolic element $g$, we denote by $g^{\pm \infty}$ the
limit points of the axis of $g$ in $\partial T$.

The \emph{translation length} $\|g\|$, of an element $g \in G$, is defined as the minimum of the distances $\d_\T(v,gv)$,
where  $v$ runs over the set of all vertices of $\T$. Since $\T$ is a simplicial tree,  $\|g\|$ is a non-negative integer.
The translation length can be used to classify the elements of $G$: $\|g\|=0$ if and only if $g$ is elliptic. And if $g$ is hyperbolic,
then the set set of all points $x$ of $\T$ such that $\d_\T(x,gx)=\|g\|$ is exactly the axis of $g$.

%If $G$ acts on the tree $\T$ without edge inversions, the edges of $\T$ can be oriented in a $G$-equivariant way.
\begin{df}\label{df:edge_transl}
Consider an edge $e$ of $\T$ with some fixed orientation.
We will say that an element $g \in G$ \emph{translates} $e$ if $g e \neq e$ and
the vertices $e_-$ and $g e_+$ lie outside of the geodesic segment $[e_+,g e_-]$.
\end{df}

Note that, in this terminology,
if $g \in G$ acts as a hyperbolic isometry and $e$ is an edge on the axis of $g$ then $e$ is translated by exactly one of the elements $g$ or $g^{-1}$
(the above definition implies that $e \ge g e$ in the sense of \cite[I.4.10]{D-D}).
Conversely, if $g \in G$ translates some edge $e$ of $\T$ then
$g$ is hyperbolic and the axis of $g$ contains the segment $[e_-,g e_+]$ (see \cite[I.4.11]{D-D}).

Our main goal in this section is to establish a criterion for the existence of hyperbolic WPD elements in $G$.
We begin with a technical result about pointwise $\e$-stabilizers defined in Section \ref{sec:eqdef}.

\begin{lemma}\label{lem:epst}
Fix any $\e\ge 0$. Let $x,y$ be any two points in $\T$ with $\d_\T (x,y)>2\e$ and let $u,v$ be any two vertices on the
geodesic segment $[x,y]$ such that $\d _\T (\{ u,v\}, \{ x,y\}) \ge \e$. Then $\pst^\e_G(\{ x,y\})$ is contained in the union of at most $2(2\e+1)$ left cosets of $\pst_G (\{u,v\})$.
\end{lemma}
\begin{proof}
Let $[u_i,v_i]$, $i=1, \ldots ,k$, be the set of all translations of $[u,v]$ by elements of $G$ such that $[u_i, v_i]\subset [x,y]$
and $\d_\T(u,u_i) \le \e$. Obviously $k\le 2(2\e +1)$ (note that edge inversions are possible).  For each $i=1, \ldots, k$, choose any element $t_i\in G$ such that $[u_i, v_i]=t_i [u,v]$.

Let $g$ be an element of $\pst^\e_G(\{ x,y\})$. That is,
\begin{equation}\label{eq:dtgx}
\max\{ \d_\T (x,g  x),\d_\T (y, g  y)\} \le\e .
\end{equation}
Let $a$ (respectively, $b$) be the point on $[x,y]$ closest to $gx$ (respectively, $gy$). Thus $[x,gx]$ is a concatenation of $[x,a]$ and $[a,gx]$ and, similarly, $[y,gy]$
is a concatenation of $[y,b]$ and $[b,gy]$.
Using (\ref{eq:dtgx}) we obtain
\begin{equation}\label{eq:dtxa}
\max \{\d_\T (x,a), \d_\T (y,b)\}\le \e
\end{equation}
and
\begin{equation}\label{eq:dtgxa}
\max \{\d_\T (gx,a), \d_\T (gy, b)\}\le \e.
\end{equation}
Since $\d_\T(x,y) > 2 \e$, (\ref{eq:dtxa}) implies that $a\ne b$ and hence $[a,b]=[x,y]\cap [g  x,g  y]$.

Since $u \in [x,y]$ and $\d_\T(u,\{ x,y\})\ge \e$, we obtain $u\in [a,b]$ using \eqref{eq:dtxa} again. Similarly we obtain $gu\in [a,b]$ and $gv\in [a,b]$ using (\ref{eq:dtgxa}).
Since  $u,g u\in [a,b]\subseteq [g x,g y]$, we have
$$
\d_\T (u,g u)= |\d_\T (u,g x)-\d_\T (g u,g x)|= |\d_\T (u,g x)-\d_\T (u,x)| \le \d_\T (x,g x)\le \e .
$$
Thus $g [u,v]=[u_i,v_i]$ for some $i\in \{ 1, \ldots , k\}$. This means that $g\in t_i\pst_G(\{u,v\})$ and hence  $\pst_G^\e(\{ x,y\})\subseteq \bigcup_{i=1}^k t_i \pst_G(\{ u, v\}).$
\end{proof}

%Note that the pointwise stabilizer of a finite set is finite whenever the stabilizer of the same set is finite. In what follows we switch from pointwise stabilizers to stabilizers to simplify out notation.

\begin{cor}\label{prop:WPD-crit} Let $G$ be a group acting on a simplicial tree $\T$ and  let
$h \in G$ be a hyperbolic element.
Suppose that for some vertices $u,v \in axis(h)$, the pointwise stabilizer $\pst_G(\{u,v\})$ is finite (the possibility $u=v$ is allowed). Then $h$ satisfies the WPD condition. In particular, either
$G$ is virtually cyclic or $G \in \X$.
\end{cor}

\begin{proof} Without loss of generality we can assume that $v \in [u,h^l u]$ for some $l \in \N$ (otherwise, exchange $u$ and $v$).
We will show that in this case $h$ satisfies Definition \ref{def:wpd} for $s=u$. Take any $\e >0$. Evidently we can increase $\e$ to assume that
$\d_\T (u,v)\le \e$. Let $m\in \mathbb N$ be such that $\d_\T (u, h^m  u)\ge \e $ and let
$x=h^{-m}u$ and $y=h^{2m}  u$. Then $[u,v] \subset [u,h^m u]\subset [x,y]$ and $\d_\T (\{ x,y\}, \{ u,v\})\ge \e$. Applying Lemma \ref{lem:epst}, we obtain that $\pst_G^\e(\{ x, y\}) $ is finite.
Since $\{ x, y\} = h^{-m}\{ u, h^{3m}u\}$,  $\pst_G^\e(\{ u, h^{3m}u\}) $ is conjugate to $\pst_G^\e(\{ x, y\}) $ in $G$, and hence it is also finite.
Thus the definition of WPD is satisfied for $s=u$ and $M=3m$.

The last claim of the corollary follows from Theorem \ref{acyl} together the observation that a hyperbolic isometry of a simplicial tree is hyperbolic (loxodromic) in the sense defined at the beginning of Subsection \ref{sec:eqdef} .
\end{proof}

\begin{rem}
The action of $\mathbb R$ on itself by translations shows that  Corollary \ref{prop:WPD-crit}  does not extend to actions on $\mathbb R$-trees.
\end{rem}

Recall that the action of $G$ on $\T $ is \emph{minimal }if $\T $ does not contain any proper $G$-invariant subtree.
Let $\widehat \T=\T \cup \partial \T$ be the natural compactification of $\T$. Take any point $x \in \T$.
The \emph{limit set} $\partial G$, of $G$ in the boundary $\partial \T$, is the intersection of the closure of $Gx$ in $\widehat \T$ with $\partial \T$. Since $G$ acts on $\T$ isometrically,
$\partial G$ does not depend on the choice of the point $x$.

\begin{lemma}\label{axisuv}
Let $\T $ be a tree with at least $3$ vertices. Suppose that a group $G$ acts minimally on $\T$ and does not fix any point of $\partial \T$.
Then for any two vertices $u,v$ of $\T$, there exists a hyperbolic element $g\in G$ such that $axis(g)$ contains $u$ and $v$.
\end{lemma}

\begin{proof}
Since $T$ has at least $3$ vertices, there exists a vertex $v$ of degree at least $2$. Let $\T_0$ be the minimal subtree of $\T $ containing the orbit $Gv$. Obviously $\T_0$ is $G$-invariant and no vertex of $\T_0$ has degree $1$ in $\T $. Since the action of $G$ is minimal, $\T=\T_0$. Thus $\T$ does not contain vertices of degree $1$. This implies that any two vertices $u,v$ of $\T$ are contained in a bi-infinite geodesic $L(u,v)$. In particular, $\T$ is unbounded.

If the action of $G$ is elliptic (i.e., orbits are bounded), then for any vertex $v$ of $\T $, the minimal subtree $\T_1$ containing the orbit $Gv$ is bounded. Obviously $\T_1$ is $G$-invariant. However this contradicts minimality of the action since $\T $ is unbounded. Thus the action of $G$ cannot be elliptic. Since $G$ does not fix any point of $\partial\T$, the action of $G$ on $\T$ cannot be parabolic or quasi-parabolic in Gromov's terminology either (see \cite[Section 8.2]{Gro}). Hence the  set $\{ g^{\pm \infty}\mid g\in \mathcal H(G)\}$, where $\mathcal H(G)$ is the set of all hyperbolic elements of $G$, is dense in $\partial G\times \partial G$
 \cite[Corollary 8.2.G]{Gro} (for a detailed proof in a more general situation see \cite{H}).
Moreover, since the action of $G$ on $\T$ is minimal, for any vertex  $v \in \T$,  $\T$ coincides with the convex span of the orbit $Gv$. Thus any point of $\T$ lies on a geodesic segment connecting two vertices from $Gv$.
It follows that $\partial G=\partial \T$, and so the  set $\{ g^{\pm \infty}\mid g\in \mathcal H(G)\}$ is dense in $\partial \T \times \partial\T$.
This implies that for any segment $I$ of $L(u,v)$, there is a hyperbolic element $g\in G$ with $I\subseteq axis(g)$, as claimed.
\end{proof}

\begin{rem}
Note that the lemma does not hold for the tree $\mathcal E$ consisting of a single edge. Indeed $\mathbb Z/2\mathbb Z$ acts on $\mathcal E$ by inversion.
The action is minimal and obviously $\mathbb Z/2\mathbb Z$ does not fix any point of $\partial \mathcal E$ as the latter is empty.
Similar counterexamples can be constructed by using any group that surjects onto $\mathbb Z/2\mathbb Z$. However if we assume that $G$ acts of $\T$ without inversions,
then the condition $|V(\T)|\ge 3$ in the lemma can be relaxed to $|V(\T)|\ge 2$.
\end{rem}

The next result is obtained in the course of proving part (3) of Proposition 6 in \cite{BF}. Note that although the authors assume that every hyperbolic element satisfies WPD there, the proof only uses the existence of a single hyperbolic WPD element.

\begin{lemma}[Bestvina-Fujiwara]\label{BF}
Let $G$ be a group acting on a hyperbolic space $X$ and containing a hyperbolic WPD element. If $G$ is not virtually cyclic, then it contains two hyperbolic elements $g_1, g_2$ such that $\{g_1^{\pm \infty}\} \cap \{ g_2^{\pm \infty}\} =\emptyset$.
\end{lemma}

We are now ready to prove the main result of this subsection.

\begin{prop}\label{prop:min-wpd}
Let $G$ be a group acting minimally on a simplicial tree $\T$. If $G$ is not virtually cyclic then the following conditions are equivalent:
\begin{enumerate}
\item[(a)] $G$ contains a hyperbolic WPD element (for the given action of $G$ on $\T$).
\item[(b)] $G$ does not fix any point of $\partial \T$ and there exist vertices $u,v$ of $\T$ such that $\pst_G(\{ u,v\})$ is finite.
\end{enumerate}
\end{prop}

\begin{proof}
Suppose that $G$ is non-elementary and contains a hyperbolic WPD element $g\in G$. Let $L$ be the axis of $g$ and let $u \in L$ be any vertex.
Let $M$ be the constant from the definition of the WPD condition for $g$ corresponding to $\e=0$. Set $v=g^M u$, then $\pst_G(\{ u,v\})$ is finite by the WPD condition.
Since the only points of $\partial \T$ fixed by a hyperbolic element $h\in G$ are $h^{\pm \infty}$, Lemma \ref{BF} implies that $G$ does not fix any point of $\partial \T$.

Conversely, suppose $G$ does not fix any point of $\partial \T$ and there exist vertices $u,v$ of $\T$ such that $\pst_G(\{ u,v\})$ is finite.
If $T$ were finite then $\pst_G(\{u,v\})$ would have finite index in $G$ implying that $|G|<\infty$, which would contradict the assumption that $G$ is not virtually cyclic. Hence the tree
$T$ must be infinite, so, by Lemma \ref{axisuv}, there exists a hyperbolic element $g\in G$ such that $axis(g)$ passes through $u$ and $v$. Therefore, according to Corollary~\ref{prop:WPD-crit},
$g$ satisfies the WPD condition.
%If $G$ was virtually cyclic, then $\langle g\rangle $ would have finite index in $G$ and hence $G$ would fix the ends of $\T$ represented by both directions of $axis (g)$.
%This contradicts our assumption and hence $G$ is not virtually cyclic.
\end{proof}

\begin{proof}[Proof of Theorem \ref{main-1}]
Suppose that $G$ is not virtually cyclic. Then it contains a hyperbolic WPD element (with respect to the action on $\T$) by Proposition \ref{prop:min-wpd}.
Hence $G$ is acylindrically hyperbolic by Theorem \ref{acyl}.
\end{proof}

\subsection{Fundamental groups of graphs of groups}\label{Sec-gg}
In  this subsection we assume that the reader is familiar with Bass-Serre theory.
For a detailed exposition we refer to \cite{Bass93,Bass76,Serre}.

Let $(\mathcal G,\ga)$ be a graph of groups, where $\ga$ is a connected graph and $\mathcal G$ is a collection consisting of
\begin{enumerate}
\item[(a)] edge groups $G_e$ and vertex groups $G_v$ associated to edges $e\in E(\ga)$ and vertices $v\in V(\ga)$ of $\Gamma $;

\item[(b)] maps $\phi_e$ and $ \phi_{\bar e}$, which for every edge $e$ of $\ga$, define the embeddings of the edge group $G_e$ into the vertex groups $G_{e_+}$ and $G_{e_-}$, respectively.
\end{enumerate}
Let also $\pi_1(\mathcal G,\ga)$ denote the fundamental group of $(\mathcal G,\ga)$. As we know, this fundamental group naturally acts on the associated Bass-Serre tree $\T$.
In this section we will investigate various conditions ensuring that this action satisfies the assumptions of Theorem \ref{main-1}.

First we will characterize the situation when $\pi_1(\mathcal G,\ga)$ fixes a point of $\partial \T$.
The proof of the following statement is essentially contained in the proof of Theorem 4.1 from \cite{A-M-Tits}:

\begin{lemma}\label{lem:fix_end} Suppose that $G$ is a group acting isometrically on a simplicial tree $\T$. If $G$ fixes a point of $\partial \T$ then all elliptic elements in $G$ form a normal subgroup $N \lhd G$
such that the quotient $G/N$ is either trivial or infinite cyclic.
\end{lemma}

\begin{proof} Let $\alpha \in \partial T$ be the point fixed by $G$ and let $x \in G$ be an elliptic element. Then $x$ fixes some vertex $u$ of $\T$, and hence it must fix (pointwise) the unique geodesic ray $[u,\alpha]$ from $u$ to $\alpha$.
If $y \in G$ is another elliptic element, then $y$ fixes the geodesic ray $[v,\alpha]$ for some vertex $v$ of $\T$. The intersection of the rays $[u,\alpha]$ and $[v,\alpha]$ is again a geodesic ray $[w,\alpha]$ for some vertex $w$.
It follows that $w$ is fixed by both $x$ and $y$, hence the product $xy$ is elliptic. Evidently an inverse and a conjugate of an elliptic element is again elliptic, therefore the set of elliptic elements forms a normal subgroup $N \lhd G$.

If $G$ has no hyperbolic elements, then $G=N$ and the lemma is proved. So, assume that $G$ has at least one hyperbolic element. Then there is a hyperbolic element $h \in G$ which has minimal translation length $\|h\|>0$.
Recall that the only fixed points of $h$ on $\partial \T$ are $h^{\pm\infty}$. After replacing $h$ with $h^{-1}$, if necessary, we can assume that $\alpha = h^{+\infty}$.
Take any other hyperbolic element $g \in G$. We want to show that $g \in \langle h \rangle N$. As above, we can suppose that $\alpha=g^{+\infty}$.
Hence the intersection of $axis (g)$ and $axis (h)$ contains an infinite geodesic ray $[v, \alpha]$ for some vertex $v$ of $\T$, and positive powers of $h$ and $g$ translate this ray into itself.

By the construction of $h$, $\|g\|\ge \|h\|$, hence there exist $m, n \in \Z $ such that $\|g\|=m \|h\|+n$ and $m >0$, $0 \le n < \|h\|$. Set $f=h^{-m}g$
and observe that $f$ translates the ray $[v,\alpha]$ into itself and $\d_\T(v,fv)=n$. So, if $n>0$ then $f$ is a hyperbolic element with $\|f\|=n<\|h\|$, contradicting the choice of $h$. Therefore $n=0$, i.e.,
$f$ fixes $v$, and so it is elliptic. Thus $g=h^m f \in h^m N$, which implies that $G=\langle h \rangle N$. Since $\langle h \rangle \cap N=\{1\}$, we see that $G/N \cong \langle h \rangle$ is infinite cyclic, and the lemma is proved.
\end{proof}

\begin{df}
Suppose that $e$ is an  edge  between vertices $u$ and $v$ in a graph of groups $(\mathcal G,\ga)$. We will say that $e$ is \emph{good} if the
natural images of $G_e$ in $G_u$ and $G_v$ are proper subgroups. Any edge that is not good will be called \emph{bad}.
\end{df}

\begin{lemma}\label{lem:ab} If the graph of groups $(\mathcal G,\ga)$ has at least one good edge then the elliptic elements of $G=\pi_1(\mathcal G,\ga)$ do not form a subgroup. In particular,
$G$ does not fix any point of $\partial \T$, where $\T$ is the corresponding Bass-Serre tree.
\end{lemma}

\begin{proof} Let $e$ be any edge of $\T$ which maps onto a good edge of $\ga = G\backslash\T$. Fixing some orientation of $e$ and abusing the notation, we let $G_e$, $G_{e_-}$ and
$G_{e_+}$ denote the $G$-stabilizers of $e$, $e_-$ and $e_+$ respectively.
Since the image of $e$ in $\ga$ is good, there is an element $a \in G_{e_-} \setminus G_e$
and an element $b \in G_{e_+} \setminus G_e$.
Clearly $a$ and $b$ are elliptic, and to prove the statement it suffices to show that the element $h=ab$ is hyperbolic. Let $f=b^{-1}e$, then
$f \neq e$ but $f_+=e_+$ as $b^{-1} \in G_{e_+} \setminus G_e$. On the other hand, the edge $h f = a e$ starts at $e_-$ and is also distinct from $e$ (see Figure \ref{fig:ab}).

\begin{figure}[ht]
\input{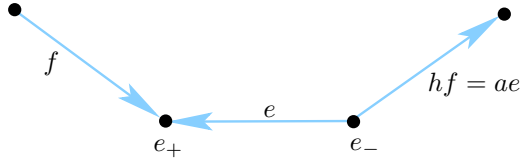}
\caption{ Illustration of Lemma \ref{lem:ab}.}\label{fig:ab}
\end{figure}

Thus we see that the edge $f$ is translated by $h$ (cf. Def. \ref{df:edge_transl}), which implies that $h$ is hyperbolic. The last claim of the lemma now follows from Lemma \ref{lem:fix_end}.
\end{proof}

\begin{df}
We will say that an oriented edge $e$ of $\ga $ is {\it reducible} if  $e_+\ne e_-$ and at least one of the maps $\phi_{\bar e}$, $\phi_{e}$ is onto.
A graph of groups $(\mathcal G,\ga)$ is \emph{reduced} if it does not contain reducible edges.
\end{df}

Thus a reducible edge is a bad edge which is not a loop. It is easy to see that
any reducible edge can be contracted without affecting the fundamental group. Therefore every finite graph
of groups can be contracted to a reduced graph of groups with the same fundamental group (see \cite{Bass76} for details, where reducible edges are called \emph{directed}).

Recall that an HNN-extension $A\ast_{C^t=D}$ of a group $A$
with associated subgroups $C$ and $D$ is called \emph{ascending} if $C=A$ or $D=A$. Ascending HNN-extension gives a basic example of a reduced graph of groups that has a
single vertex and a single bad loop at this vertex.
The following proposition characterizes the case when the fundamental group of a reduced graph of groups fixes a point on the boundary of the corresponding Bass-Serre tree:

\begin{prop}\label{prop:charact} Suppose that $(\mathcal G,\ga)$ is a reduced graph of groups with at least one edge.
Let $G=\pi_1(\mathcal G,\ga)$ and let $\T$ be the corresponding Bass-Serre tree. Then the following are equivalent:

\begin{itemize}
\item[(i)] $G$ fixes a point of the boundary $\partial \T$;
\item[(ii)] $(\mathcal G,\ga)$ has only one vertex $v$ and one (unoriented) bad edge $e$ with $e_-=e_+=v$. In other words, $G$ is an ascending HNN-extension of $G_v$.
\end{itemize}
\end{prop}

\begin{proof} The fact that (ii) implies (i) is well-known and we leave it as an exercise for the reader.
Now, suppose that (i) holds. Then $(\mathcal G,\ga)$ has no good edges by Lemma \ref{lem:ab}. Since $(\mathcal G,\ga)$ has no reducible edges
by the assumptions and $\ga$ is connected, one concludes that $\ga$ has only one vertex $v$ and all the edges of $\ga$ are loops at $v$.

The quotient of $G=\pi_1(\mathcal G,\ga)$ by the normal closure $N \lhd G$, of $G_v$ in $G$, is the free group, whose rank is equal to the number of (unoriented) loops at $v$.
Since $v$ is the only vertex of $\ga$, every elliptic element is conjugate to an element of $G_v$ in $G$. So, by Lemma \ref{lem:fix_end}, $N$ is the set of all elliptic elements, and
the quotient $G/N$ is cyclic. Thus $\ga$ can have at most one loop at $v$, which is a bad edge (as defined above), i.e., (ii) holds.
\end{proof}

For a general graph of groups  $(\mathcal G,\ga)$,  the situation when the canonical action of $G=\pi_1(\mathcal G,\ga)$  on the corresponding Bass-Serre tree $\T$ is minimal
was characterized by Bass in \cite{Bass93}. The next lemma follows from \cite[Prop. 7.12]{Bass93}:
%Using his terminology, we shall say that a vertex $v$ of $(\mathcal G,\ga)$ is \emph{terminal} if there is only one oriented
%edge $e$ in $\ga$ with $e_+=v$ (in particular $e$ cannot be a loop) and the homomorphism $\phi_e:G_e \to G_v$ is surjective.
%Evidently, if $(\mathcal G,\ga)$ does have a terminal vertex, then this vertex together with the adjacent edge can be removed without affecting the fundamental group.

\begin{lemma}\label{lem:min-act}
Let $(\mathcal G,\ga)$ be a  reduced graph of groups of finite diameter.
Then the action of $\pi_1(\mathcal G,\ga)$ on the corresponding Bass-Serre tree is minimal.
\end{lemma}

A combination of Proposition \ref{prop:charact} with Lemma \ref{lem:min-act} and Theorem \ref{main-1} gives a criterion to show that the fundamental group
of a finite reduced graph of groups is acylindrically hyperbolic (see Theorem \ref{thm:reduced-crit} below). Before formulating it we need to make a couple of observations.

\begin{rem}\label{rem:vert-edg} Suppose that $G$ is a group acting on a simplicial tree $\T$.
If the tree $\T$ has at least two vertices, the existence of  vertices $u,v \in \T$ with $|\pst_G(\{u,v\})|<\infty$ is equivalent to the existence
of two edges $e,f$ of $\T$ such that $|\pst_G(\{e,f\})|<\infty$. (Note that we allow the possibilities $u=v$ and/or $e=f$.)
\end{rem}

Indeed,
in one direction this is evident. For the other direction, suppose that $e$ and $f$ are  edges of $\T$ with $|\pst_G(\{e,f\})|<\infty$.
Without loss of generality we can assume that the segment $[e_-,f_+]$ contains both of these edges. Then
$\pst_G(\{e_-,f_+\})=\pst_G([e_-,f_+])=\pst_G(\{e,f\})$, i.e., one can take $u=e_-$ and $v=f_+$.

\begin{rem}\label{rem:fin_pst} Suppose that $G$ is the fundamental group of a graph of groups $(\mathcal G,\ga)$ and $\T$ is the corresponding Bass-Serre tree.
In algebraic terms, the existence of vertices $u,v$ of $\T$ such that $\pst_G(\{ u,v\})$ is finite means that
$|G_{u'}^g\cap G_{v'}|<\infty$ for some $g\in G$ and some vertices $u'$, $v'$ of $\ga$. Similarly, the
existence of two edges $e,f$ of $\T$ such that $|\pst_G(\{e,f\})|<\infty$ corresponds to
$|G_{e'}^g\cap G_{f'}|<\infty$ for some $g\in G$ and some edges $e'$, $f'$ of $\ga$.
\end{rem}

\begin{thm}\label{thm:reduced-crit}
Let $G$ be the fundamental group of a finite reduced graph of groups $(\mathcal G,\ga)$ with at least one edge such that
the condition (ii) from Proposition \ref{prop:charact} does not hold.
Suppose that there are edges $e,f$ of $\ga$ (not necessarily distinct) and an element $g \in G$ such that $|G_f \cap G_e^g|<\infty$.
Then $G$ is either virtually cyclic or acylindrically hyperbolic.
\end{thm}

\begin{proof}
Let $\T$ be the Bass-Serre tree associated to $(\mathcal G,\ga)$. By Lemma \ref{lem:min-act}, the action of $G$ on $\T$ is minimal, and, by Proposition \ref{prop:charact},
$G$ does not fix any point of $\partial \T$. Since $(\mathcal G,\ga)$ is reduced and has at least one edge, it is easy to see that the tree $\T$ must be infinite
(any loop gives rise to a hyperbolic element; if there are no loops, then all the edges must be good, hence a hyperbolic element can be produced as in the proof Lemma \ref{lem:ab}).
By Remarks \ref{rem:fin_pst} and \ref{rem:vert-edg}, the assumptions imply that $\T$ contains two vertices $u,v$ such that $\pst_G(\{u,v\})$ is finite.
Therefore,  we can apply Theorem \ref{main-1} to conclude that
$G$ is either virtually cyclic or acylindrically hyperbolic.
\end{proof}

Corollary \ref{cor:amalg-intr} from Section \ref{sec:results} is an immediate consequence of Theorem \ref{thm:reduced-crit} applied in the situation when $(\mathcal G,\ga)$ consists of two vertices
and a good edge connecting them. Corollary~\ref{cor:HNN-intr} corresponds to the case when $(\mathcal G,\ga)$ has one vertex and one (good) loop at this vertex. Thus
its claim also follows from Theorem \ref{thm:reduced-crit} modulo the observation that a non-ascending HNN-extension of any group contains
non-abelian free subgroups (see, for example, \cite[Thm. 6.1]{Bass76}), and so it cannot be virtually cyclic.

It is easy to see that the condition $A\ne C\ne B$ is necessary in Corollary \ref{cor:amalg-intr}.
As for the requirement $C\ne A\ne D$ in Corollary \ref{cor:HNN-intr}, the situation is more complicated and counterexamples are not so obvious. We note the following.

\begin{prop}\label{prop:asc-HNN}
Let $G$ be an ascending HNN-extension of a group $A$. If $G$ is acylindrically hyperbolic, then so is $A$.
\end{prop}

\begin{proof}
Let $t$ be the stable letter of the HNN-extension, $t^{-1}At \le A$. Then the subgroup $N=\bigcup\limits_{n=1}^\infty t^n At^{-n}$ is normal in $G$ and $G/N\cong \langle t\rangle$. Consequently we have $G=\langle t\rangle A\langle t\rangle $. Thus by part (d) of Theorem \ref{thm:elem-prop-ah} $A$ must be acylindrically hyperbolic.
\end{proof}

Proposition \ref{prop:asc-HNN} allows to construct examples of ascending HNN-extensions with weakly malnormal base which are not acylindrically hyperbolic.

\begin{ex}\label{ex:HNNB}
Using methods from \cite[Section 4.2]{OS} it is not hard to construct a proper malnormal subgroup $B<A$ of a free Burnside group $A$ of rank $2$ and large odd exponent such that $B\cong A$. Let $G$ be the corresponding ascending HNN-extension of $A$. Then $A$ and $B$ are weakly malnormal in $G$, but $G$ is not  acylindrically hyperbolic by a combination of Proposition \ref{prop:asc-HNN} and Example \ref{ex:non-ah} (b).
\end{ex}

If $A$ is acylindrically hyperbolic, then its ascending HNN-extensions can be acylindrically hyperbolic as well. For example, many ascending HNN-extensions of free groups are hyperbolic. Nevertheless the example below shows that acylindrical hyperbolicity (even relative hyperbolicity) of $A$  together with weak malnormality is still not sufficient to derive acylindrical hyperbolicity of $G$.

\begin{ex}\label{ex:bg}
In the proof of \cite[Prop. 18]{O-asdim}, the second author constructed a boundedly generated finitely presented group (denoted by $G$ in \cite{O-asdim}) which is universal, i.e., contains an isomorphic
copy of every recursively presented group. Let us denote this group by $W$. Let $A=W\ast \mathbb Z$ and let $s\in A$ be a generator of the subgroup $\mathbb Z$.
Then $A$ is acylindrically hyperbolic (in fact, it is hyperbolic relative to $W$). Since $A$ is finitely presented,
it is isomorphic to a subgroup $B\le W$. Let $G$ be the corresponding ascending HNN-extension of $A$ with the stable letter $t$. Obviously $B^s\cap B =\{1\}$ and hence $A$ and $B$ are weakly malnormal in $G$.
Arguing as in the proof of Proposition \ref{prop:asc-HNN}, we obtain $G=\langle t\rangle W\langle t\rangle $. Since $W$ is boundedly generated, so is $G$. Hence $G$ is not acylindrically hyperbolic by part (d) of Theorem \ref{thm:elem-prop-ah}.
\end{ex}

\subsection{One-relator groups}\label{sec:1-rel}
We begin with the proof of Corollary \ref{cor:one-rel}, which is inspired by \cite{S} and \cite{SS}.

\begin{proof}[Proof of Corollary \ref{cor:one-rel}]
It is proved in \cite{SS} (see the proof of Theorem III there) that $G$ splits as $G=H\ast _{A^t=B}$, where $A, B$ are proper subgroups
of $H$, and there exists $h\in H$ such that $A\cap B^h=\{ 1\}$. Thus the claim follows from Corollary \ref{cor:amalg-intr}
and the well known fact that a one relator group with at least $3$ generators is never virtually cyclic.
\end{proof}

Our next goal is to study one-relator groups with $2$ generators. As we already mentioned in Example \ref{ex:non-ah-1},
the Baumslag-Solitar groups $BS(m,n)$ are not acylindrically hyperbolic unless $m=0$ or $n=0$. The obstacle for acylindrical hyperbolicity of
$BS(m,n)$ is the existence of an infinite cyclic $s$-normal subgroup. The are are many other one-relator groups with the same property.
However, we will show that all of them (as well as other potential examples of non-acylindrically hyperbolic one-relator groups) have low complexity in a certain sense.

Recall that, given a group $H$ and an injective (but not necessarily surjective) endomorphism $\phi: H \to H$, the\emph{ mapping torus of $\phi$} is the (ascending)
HNN-extension $$G= \langle H,t \, \mid \, tht^{-1}=\phi(h)~\forall h \in H \rangle.$$

\begin{prop}\label{cor:1-rel-2-gen} Let $G$ be a group with two generators and one defining relator. Then at least one of the following holds:
\begin{itemize}
  \item[(i)] $G$ is acylindrically hyperbolic;
  \item[(ii)] $G$ contains an infinite cyclic $s$-normal subgroup.
  More precisely, either $G$ is infinite cyclic or it is an HNN-extension of the form $$G=\langle a,b, t\mid  a^t=b, r=1\rangle $$ of a one-relator group
  $H= \langle a,b \mid r\rangle $ with non-trivial center, so that $a^k=b^l$ in $H$ for some $k,l \in \Z \setminus\{0\}$. In the latter case $H$  is (finitely generated free)-by-cyclic and contains a finite index normal subgroup
splitting as a direct product  of a finitely generated free group with an infinite cyclic group.
  \item[(iii)] $G$ is isomorphic to the mapping torus of an injective endomorphism of a finitely generated free group.
\end{itemize}
Moreover, the possibilities (i) and (ii) are mutually exclusive.
\end{prop}

In order to prove Proposition \ref{cor:1-rel-2-gen} we will need the following lemma:

\begin{lemma}\label{lem:inter_conj_free_gps} If $C$ and $D$ are finitely generated subgroups of infinite index in a finitely generated free group $A$ then there exists an element $f \in A$
such that $fDf^{-1} \cap C=\{1\}$.
\end{lemma}

\begin{proof} One way to prove this would be to refer, first, to a result of W. Neumann \cite[Prop.~2.1]{WN} stating that there are finitely many elements $g_1,\dots,g_n \in A$ such that if
$D^g \cap C \neq \{1\}$ then $g \in Cg_i D$ for some $i \in \{1,\dots,n\}$. Second, one can use the fact that the free group $A$ cannot be covered by finitely many double cosets from
$C \backslash A / D$ because $C$ and $D$ are finitely generated and have infinite index in $A$ (a generalization of this fact for quasiconvex subgroups of hyperbolic groups
was proved in \cite[Cor. 1]{Min-2}). This implies that there is some $f \in A$ such that $f \notin \bigcup_{i=1}^n Cg_i D$, and hence $D^f \cap C=\{1\}$.

However, let us give a  direct proof of the lemma using Stallings core graphs (see \cite{K-M} for an introduction to this method). Suppose that $A$ has a free generating set $X$ of cardinality $k \in \N$.
Let $\ga_1$ and $\ga_2$ be the finite $X$-labelled directed core
graphs corresponding to the subgroups $C$ and $D$ of $A$ respectively (cf. \cite[Prop. 3.8]{K-M}). Thus each $\ga_i$ has a distinguished vertex $v_i$ such that the language of reduced
closed loops at this vertex coincides with $C$  if $i=1$ and with $D$ if $i=2$. Since $C$ and $D$ have infinite index in $A$, each $\ga_i$ possesses a vertex of degree strictly less than $2k$
(see \cite[Prop. 8.3]{K-M}).
After replacing $C$ and $D$ with their conjugates we can assume that these vertices are $v_i$, $i=1,2$. Thus $v_1$ has no outgoing edge labelled by $x_1$ for some
$x_1\in X^{\pm 1}$, and $v_2$ has no incoming edge labelled by $x_2$ for some $x_2\in X^{\pm 1}$. Now, take a freely reduced word $w$ over $X^{\pm 1}$ such that $w$ starts with $x_1$ and ends with $x_2$.
Let $p$ be the simple directed path (regarded as a simplicial graph) labelled by $w$.
Construct the new $X$-labelled directed graph $\ga_2'$ by attaching the endpoint of $p$ to the graph $\ga_2$ at $v_2$. Let $v_2'$ be the initial vertex of $p$.
Observe that, by construction, the graph $\ga_2'$ is folded, and so, equipped with the base vertex $v_2'$, it represents the core graph of the subgroup $fDf^{-1}$ in $A$, where $f\in A$ is the element given by the word $w$.
Moreover, the label of a non-trivial reduced loop at $v_2'$ in $\ga_2'$ cannot be equal to the label of a reduced loop at $v_1$ in $\ga_1$, because the former starts with the letter $x_1$ and $v_1$ has no
outgoing edges with that label. Therefore $C \cap D^f=\{1\}$, as required.
\end{proof}

\begin{proof}[Proof of Proposition \ref{cor:1-rel-2-gen}]
If $G$ is cyclic, (ii) holds, so we suppose that $G$ is not cyclic. Thus $G=\langle a,t \mid r \rangle$, where $r$ is a cyclically reduced word over the alphabet $\{a,t\}^{\pm 1}$.
Without loss of generality we can assume that the word $r$ is non-empty and
 the total exponent of $t$ in $r$ is $0$ by \cite[Lemma 4.1]{SS}. Then, using the standard Magnus-Moldavanski\v{\i} re-writing procedure (see, for example, \cite[p. 738]{SS})
one shows that $G$ has a presentation
\begin{equation}\label{eq:G-pres}
G=\langle a_1,\dots,a_n,t  \mid  r', ta_1t^{-1}=a_2,\dots,ta_{n-1}t^{-1}=a_n \rangle,
\end{equation}
where $n \in \N$ and  $r'$ is a cyclically reduced word in the alphabet $\{a_1,\dots,a_n\}^{\pm 1}$ involving both $a_1$ and $a_n$.
Let $H$ be the one-relator group given by the presentation $H=\langle  a_1,\dots,a_n  \mid r' \rangle$.

If $n=1$, then $G \cong H*\langle t \rangle$. If $H = \{1\}$, $G$ is infinite cyclic. If $H$ is non-trivial but finite, $G$ is non-elementary hyperbolic and, hence, it is acylindrically hyperbolic.
Finally, if $H$ is infinite, $G$ is also acylindrically hyperbolic as $H$ is a non-degenerate and hyperbolically embedded subgroup in $G$ (see Example \ref{hes} and Theorem \ref{acyl}).

Suppose, now, that $n \ge 2$.
By construction, both $a_1$ and $a_n$ occur in $r'$. Consequently, using the Magnus's Freiheitssatz, the elements $\{ a_1,\dots,a_{n-1}\}$ and $\{ a_2,\dots,a_{n}\}$
freely generated free subgroups of rank $n-1$ in $H$. The presentation \eqref{eq:G-pres} now shows that $G$ is the HNN-extension of
$H$ with the associated Magnus subgroups $A:=\langle a_1,\dots,a_{n-1}\rangle\leqslant H$ and
$B:=\langle a_2,\dots,a_{n}\rangle\leqslant H$. Again we consider two possibilities.

First, let us consider the case $n=2$. Then $A=\langle a_1 \rangle$ and $B=\langle a_2 \rangle$ are infinite cyclic groups. If $A \cap B=\{1\}$ then $G \in \X$ by Corollary \ref{cor:HNN-intr} and Remark \ref{rem:HNN}.
So, assume that $ a_1^k= a_2^l$ in $H$ for some $k,l \in \Z\setminus\{0\}$. Then $a_1^k$ is central in $H$ and $(a_1^l)^t=a_2^l=a_1^k$, which easily implies that $A$ is $s$-normal in $G$.
It also follows that $H/[H,H]\not\cong \mathbb Z\oplus \mathbb Z$. Therefore $H$ can be represented as a stem product of cyclic groups by a theorem of Pietrowski \cite[Thm. 1]{Piet}.
Moreover, in \cite{Mur} Murasugi showed that $H$ is (finitely generated free)-by-cyclic in this case (see the first paragraph of the proof of Theorem 1 in \cite{Piet}).
Thus $H$ contains a finitely generated normal free subgroup $F \lhd H$ and an element $t \in H$ of infinite order so that $F \cap \langle t \rangle =\{1\}$; in other words,
$H = F \rtimes \langle t \rangle$. If $F$ is trivial or infinite cyclic then $H$ is virtually abelian, and thus (ii) holds. So, assume that $F$ is non-cyclic.
Then it has trivial center, and since the center of $H$ is non-trivial, there must exist $p \in \Z\setminus\{0\}$ and $f \in F$ such that the element $ft^p$ is central in $H$.
Observe that the subgroup $P:=\langle F,t^p \rangle$ is normal in $H$ and has index $|p|$. Clearly $$P=F \cdot \langle t^p \rangle =  F \cdot \langle ft^p \rangle \cong F \times \langle ft^p \rangle \cong F \times \Z.$$
Hence (ii) holds.

Thus we can further assume that $n \ge 3$.
Consider the subgroups $C=A \cap B$ and $D=t^{-1} (A \cap B) t$ in $G$. Then $C,D \leqslant A$ as $A\cap B \leqslant B$ and $A=t^{-1}Bt$ in $G$.
If  $D=A$ then $A \cap B=B$, i.e., $B \leqslant A$ in $H$, which implies that $H=A$ and hence $G$ is isomorphic to the mapping torus of some injective endomorphism of the free group $A$. Similarly, (iii) holds if $C = A$.

It remains to deal with the case when $C,D$  are proper subgroups of $A$ and $n \ge 3$. By a result of Collins  \cite{Collins} the intersection $A \cap B$, of the Magnus subgroups of $H$ is a free subgroup of rank at most $n-1$.
Thus $C$ and $D$ are proper subgroups of the free group $A$ and their rank  does not exceed the rank of $A$. Since the rank of $A$ is $n-1 \ge 2$, the Schreier index formula (cf. \cite[I.3.9]{L-S})
implies that the rank of any proper finite index subgroup of $A$ is strictly greater than $n-1$.
Hence both $C$ and $D$ must have infinite index in $A$. Consequently, we can apply Lemma \ref{lem:inter_conj_free_gps} to find an element $f \in A$ such that $D^f \cap C =\{1\}$ in $A$.
Set $g=ft^{-1} \in G$. We will now check that $A^g \cap B=\{1\}$, which would yield that $G \in \X$ by Corollary \ref{cor:HNN-intr} and Remark \ref{rem:HNN}, finishing the proof.

Indeed, suppose that $A^g \cap B\neq \{1\}$, i.e., there exist $h \in A\setminus\{1\}$ and $b \in B\setminus \{1\}$ such that
$ft^{-1} h t f^{-1}b^{-1}=1$ in $G$. By Britton's Lemma (see \cite[IV.2]{L-S}), $h \in B$, hence $t^{-1}ht \in D$.
Therefore $ft^{-1} h t f^{-1}=b \in D^f \cap B$, but $D^f \leqslant A$ and so $$D^f \cap B=D^f \cap (A\cap B)=D^f \cap C=\{1\},$$ implying that $b=1$.
The above contradiction shows that $A^g \cap B=\{1\}$, as claimed.

Finally we observe that conditions (i) and (ii) are indeed mutually exclusive, because the first claim in (ii) implies that $G$ is not acylindrically hyperbolic by Theorem \ref{thm:elem-prop-ah}.
\end{proof}

\subsection{Automorphism groups of polynomial algebras}\label{sec:GA}
Let $k$ be a field. We denote by ${\rm Aut}\,k[x_1, \ldots, x_n]$ the group of $k$-automorphisms of the (formal, commutative)
polynomial algebra $k[x_1, \ldots, x_n]$, i.e., those automorphisms of the ring $k[x_1, \ldots, x_n]$ that fix $k$. This group is sometimes called the integral Cremona group of dimension $n$ over $k$.

Note that every $k$-automorphism $\alpha$ of $k[x_1, \ldots, x_n]$ is uniquely determined by the images $\alpha (x_1), \ldots , \alpha (x_n)$. Thus elements of ${\rm Aut}\,k[x_1, \ldots, x_n]$ can be represented by polynomial maps $$\left(\begin{array}{c}
                                                                                                        x_1 \\
                                                                                                        \vdots \\
                                                                                                        x_n
                                                                                                      \end{array}\right)
 \mapsto \left(
                   \begin{array}{c}
                     p_1 (x_1, \ldots, x_n)\\
                     \vdots \\
                     p_n (x_1, \ldots, x_n)\\
                   \end{array}
                 \right) ,
$$
where $p_i (x_1, \ldots, x_n)\in k[x_1, \ldots, x_n]$ for $i=1, \ldots, n$. Obviously multiplication in ${\rm Aut}\,k[x_1, \ldots, x_n]$ corresponds to the composition of such maps.

It is easy to see that ${\rm Aut}\,k[x]$ is generated by the affine maps $x\mapsto ax +b$, where $a\in k^\ast$ and $b\in k$. Therefore, ${\rm Aut}\,k[x] \cong k \rtimes k^\ast$, where $k^\ast $ acts on the additive group $k$ by multiplication. In particular,  ${\rm Aut}\,k[x]$ is metabelian and hence it is not acylindrically hyperbolic.

For $n=2$, the situation is completely different. We recall a well-known decomposition theorem for ${\rm Aut}\,k[x,y]$. Let $A$ be the group of all affine maps $$\left(                                                                                                                                                              \begin{array}{c}
x \\
y \\
\end{array}
\right)
\mapsto
\left(\begin{array}{c}
ax+by+c \\
dx+ey+f \\
\end{array}
\right),
$$
where $a, \ldots, f\in k$ and $\left| \begin{array}{cc}
                a & b \\
                c & d
              \end{array}\right|\ne 0$. Let also $B$ be the so-called triangular group consisting of all maps of the form
\begin{equation}\label{triang}
\left(                                                                                                                                                              \begin{array}{c}
x \\
y \\
\end{array}
\right)
\mapsto
\left(\begin{array}{c}
ax+b \\
cy+g(x) \\
\end{array}
\right),
\end{equation}
where $a,c\in k^\ast$, $b\in k$, and $g(x)\in k[x]$. It is easy to see that the intersection $C=A\cap B$ consists of all maps of the form (\ref{triang}), where $g$ is a linear polynomial. It is known that $${\rm Aut}\,k[x,y]=A\ast _CB.$$

The fact that ${\rm Aut}\,k[x,y]$ is generated by $A$ and $B$ was proved by Jung \cite{Jung} for ${\rm char}\, k=0$ and van der Kulk \cite{vdK} in the general case. The above amalgamated product decomposition can be derived from their work quite easily. Alternative proofs were later suggested by many people including Dicks \cite{Dic}, Nagata \cite{Nag}, and Wright \cite{Wri}.

For $k=\mathbb C$ the lemma below can be extracted from \cite{FL}. It is quite possible that similar (rather sophisticated) arguments work for any field. We provide an elementary ``hands on" proof.

\begin{lemma}\label{aCaC}
There exists $\alpha \in {\rm Aut}\, k[x,y]$ such that $\alpha^{-1} C \alpha \cap C=\{ 1\}$.
\end{lemma}

\begin{proof}
Our choice of  $\alpha$ will depend on the characteristic of $k$. We first consider the case when ${\rm char}\, k\ne 2$.

Let $\alpha$ be the map $$ \alpha: \left(                                                                                                                                                              \begin{array}{c}
x \\
y \\
\end{array}
\right)
\mapsto
\left(\begin{array}{c}
x+y^2 \\
y+p(x+y^2) \\
\end{array}
\right),
$$ where $p(t)=t^4+t^2$. It is obvious that $\alpha \in {\rm Aut}\, k[x,y]$.
Let $\gamma \in \alpha^{-1} C \alpha \cap C$.
Then $\gamma $ has the form  $$\gamma: \left(                                                                                                                                                              \begin{array}{c}
x \\
y \\
\end{array}
\right)
\mapsto
\left(\begin{array}{c}
ax +b \\
cx+dy+e \\
\end{array}
\right) ,
$$
and
\begin{equation}\label{ag=ga}
\alpha \gamma=\gamma ^\prime\alpha
\end{equation}
for some $\gamma ^\prime \in C$,
$$\gamma^\prime\colon \left(                                                                                                                                                              \begin{array}{c}
x \\
y \\
\end{array}
\right)
\mapsto
\left(\begin{array}{c}
a^\prime x +b^\prime \\
c^\prime x+d^\prime y+e^\prime \\
\end{array}
\right).
$$

Computing the compositions in (\ref{ag=ga}), we obtain
\begin{equation}\label{ag1}
ax +b +(cx+dy+e)^2 = a^\prime (x+y^2) + b^\prime
\end{equation}
and
\begin{equation}\label{ag2}
cx+dy+e +p (ax+b+(cx+dy+e)^2) = c^\prime(x+y^2) + d^\prime(y +p(x+y^2))+e^\prime.
\end{equation}
Equating the coefficients at $y$ in (\ref{ag2}), we obtain $d^\prime =d$. Note that if $d=d^\prime=0$, then (\ref{ag2}) yields $c^\prime =0$, which implies that $\gamma^\prime $ is not an automorphism.
Thus $d\ne 0$. Now using (\ref{ag1}), we obtain $c=0$ and then $a=a^\prime=d^2$ and $2de=0$. The later equality in turn gives $e=0$ as $d\ne 0$ and ${\rm char}\, k\ne 2$. Using these we can rewrite (\ref{ag2}) as
$$
p(ax+b+(cx+dy+e)^2)=c^\prime(x+y^2) + dp(x+y^2)+e^\prime.
$$
Further replacing $ax+b+(cx+dy+e)^2$ with the right side of (\ref{ag1}) and denoting $x+y^2$ by $t$, we obtain
\begin{equation}\label{ag3}
p(d^2t +b^\prime) = c^\prime t +dp(t) +e^\prime.
\end{equation}
Now equating the coefficients at $t^3$ in (\ref{ag3}) we obtain $b^\prime =0$. Substituting this to (\ref{ag3}) and expanding $p$ yields
\begin{equation}\label{ag4}
d^8t^4+ d^4t^2= c^\prime t +dt^4 +dt^2 +e^\prime.
\end{equation}
Thus we have $d^7=1$ and $d^3=1$,  which in turn implies $d=1$. As $a=d^2$, we obtain $a=1$. Finally substituting $b^\prime =0$ and $e=0$ in (\ref{ag1}), we obtain $b=0$. Thus $\gamma $ is the identity.

Now assume that ${\rm char}\, k=2$. Then we let $\alpha$ be the map $$ \alpha: \left(                                                                                                                                                              \begin{array}{c}
x \\
y \\
\end{array}
\right)
\mapsto
\left(\begin{array}{c}
x+y^3 \\
y+p(x+y^3) \\
\end{array}
\right),
$$ where $p(t)=t^5+t^2$. Arguing as above we obtain $\gamma = id$. This case is completely analogous to the previous one and we leave it to the reader.
\end{proof}

\begin{proof}[Proof of Corollary \ref{GA2}]
Clearly $C \neq A$ and $|B:C|=\infty$. Hence ${\rm Aut}\, k[x,y] =A*_C B$ contains non-abelian free subgroups (cf. \cite[Thm. 6.1]{Bass76}) and therefore it
is not virtually cyclic. Thus the claim
follows from Corollary~\ref{cor:amalg-intr} and Lemma \ref{aCaC}.
\end{proof}

\subsection{Miscellaneous examples}\label{sec:misc}

\begin{cor}\label{cor:am-hyp}
Let $G=A\ast_{C} B$ ($G=A\ast_{D=C^t}$), where $A$ is a hyperbolic group, $C$ is a quasi-convex subgroup of infinite index in $A$, and $B$ is arbitrary (respectively, $D\lneqq A$).
Then $G$ is either virtually cyclic or acylindrically  hyperbolic.
\end{cor}

\begin{proof} In the case when $G=A \ast_C B$, if $C=B$ then $G=A$ is hyperbolic, and hence $G \in \X$ unless it is virtually cyclic. So, we can further assume that $C \neq B$.

Recall that a quasi-convex subgroup of a hyperbolic group has finite width -- see \cite[Main Thm.]{GMRS}. This means that there is $n \in \N$ such that for any collection of elements $a_1,\dots,a_{n+1} \in A$,
with $Ca_i \neq Ca_j$ whenever $i \neq j$, there exist $k,l \in \N$, $1 \le k < l \le n+1$, such that $|a_k^{-1} C a_k \cap a_l^{-1} C a_l|<\infty$. Since $|A:C|=\infty$, a collection of elements
$a_1,\dots,a_{n+1}$ with these properties exists, therefore $C$ is weakly malnormal in $A$ (and, hence, in $G$). Thus to reach the required conclusion it remains to apply
Corollary~\ref{cor:amalg-intr} in the case of the amalgamated product or Corollary \ref{cor:HNN-intr} in the case of the HNN-extension.
\end{proof}

It is not hard to show that in Corollary \ref{cor:am-hyp} $G$ will be virtually cyclic only if $A$ is virtually cyclic and $C=B$.
Let us mention one particular case, which may be compared with simple groups discussed in Example \ref{ex:non-ah-1}.(a),  showing that neither the assumption that $C$ is finitely generated
nor the assumption that $C$ has infinite index in $A$ can be dropped.

\begin{cor}\label{cor:am-free}
Let $G=A\ast_{C} B$, where $A$ is a free group and $C \leqslant A$ is a finitely generated subgroup of infinite index.
Then $G$ is either infinite cyclic or acylindrically  hyperbolic.
\end{cor}

Obviously, under the assumptions of the corollary, $G$ is cyclic if and only if so is $A$ and $C=B$. It is worth mentioning that non-simplicity of groups from Corollary \ref{cor:am-free}
with $B$ free was proved in \cite{S} using the small cancellation theory. In this particular case the assumptions of Corollary \ref{cor:amalg-intr} can be easily
verified using the following theorem of M. Hall: \emph{every finitely generated subgroup of a free group $F$ is a free factor of a finite index subgroup of $F$.}

It seems plausible that Corollary \ref{cor:am-hyp} generalizes to amalgamated products and
HNN-extensions of relatively hyperbolic and even acylindrically hyperbolic groups (for a suitable definition of quasi-convex subgroups). We leave such generalizations to the reader.

Schupp's arguments from \cite{S} combined with Corollary  \ref{cor:amalg-intr} also yield the following:

\begin{cor}
The Higman group $$G=\langle a,b,c,d \mid a^b=a^2, \, b^c=b^2,\, c^d=c^2,\, d^a=d^2\rangle $$ is acylindrically hyperbolic.
\end{cor}

%\begin{proof}
%...
%\end{proof}

%%%%%%%%%%%%%%%%%%%%%%%%%%%%%%%%%%%%%%%%%%%%%%%%%%%%%%%%%%%%%%%%%%%%%%%%%%%%%%%

\section{$3$-manifold groups}

%%%%%%%%%%%%%%%%%%%%%%%%%%%%%%%%%%%%%%%%%%%%%%%%%%%%%%%%%%%%%%%%%%%%%%%%%%%%%%%

\subsection{Preliminary information on the JSJ decomposition}
We begin by reviewing some results about $3$-manifolds and their fundamental groups used in this paper. For basic definitions we refer to \cite{Hem}. Our main reference for properties of fundamental groups of $3$-manifold groups is \cite{AFW}. Details about Seifert fibered manifolds can be found in \cite{Bri,Sco83}.

Let $M$ be a closed, orientable, irreducible $3$-manifold. The Jaco-Shalen-Johannson decomposition theorem provides a canonical finite collection of
disjoint incompressible tori $T=\{T_i\}$ in $M$ such that all connected components of $M\setminus T$ are either a Seifert fibred or atoroidal. A minimal such collection is unique up to isotopy. By the Seifert-van Kampen theorem, this decomposition of $M$ induces a decomposition of $\pi _1(M)$ into a graph of groups. We will call it the \emph{JSJ decomposition} of $\pi _1(M)$.

Recall that the action of a group $G$ on a tree $\T $ is \emph{$k$-acylindrical in the sense of Sela}, if the diameter of the fixed point set of every $g\in G\setminus \{ 1\}$ is at most $k$.
For example, the action of a free product $G=A\ast B$ on the associated Bass-Serre tree is $1$-acylindrical.
The following result was proved by Wilton and Zalesskii in \cite{WZ}.

\begin{lemma}[{\cite[Lemma 2.4]{WZ}}]\label{WZ}
Let $M$ be a closed, orientable, irreducible $3$-manifold. Then either $M$ has a finite-sheeted covering space that is a torus bundle over a circle or
the action of $\pi_1(M)$ on the Bass-Serre tree associated to the JSJ decomposition of $\pi_1(M)$ is $4$-acylindrical.
\end{lemma}

We now briefly discuss the pieces of the JSJ decomposition. By the Perelman's proof of the Thurston's Geometrization Conjecture \cite{Per1}--\cite{Per3}, atoroidal pieces admit a finite volume hyperbolic structure
(in the case of Haken manifolds this was originally proved by Thurston -- see \cite[Thm. 1.42]{M_Kapovich}).
Hence the fundamental group of each atoroidal piece is hyperbolic relative to peripheral subgroups isomorphic to $\mathbb Z\oplus \mathbb Z$ \cite{Farb}.
As for a Seifert fibered piece $S$, we will just use the fact that if the  fundamental group $\pi_1(S)$ is infinite then it fits into the short exact sequence
$$\{1\} \to Z \to \pi_1(S)\to H\to \{1\},$$
where $Z$ is infinite cyclic and $H$ has a finite index subgroup which is a surface group (here by a surface group we mean any group isomorphic to the fundamental group of a compact surface) -- see \cite[Sec. 1.6]{M_Kapovich}.

\subsection{A trichotomy for $3$-manifold groups}
We say that an action of a group $G$ on a tree is $\T$ \emph{acylindrical in the sense of Sela} if it is $k$-acylindrical for some $k\ge 0$.
We begin with a lemma that relates the notion to acylindricity in the sense of Sela to the one defined in the Introduction.

\begin{lemma}\label{lem:Sa}
If the action of a group $G$ on a tree $\T$  is acylindrical in the sense of Sela, then it is acylindrical (as defined in the Introduction).
\end{lemma}

\begin{proof} Suppose that the action is $k$-acylindrical for some $k\ge 0$.
Fix any $\e>0$ and consider any two points $x,y$ on $\T $ such that $\d_\T (x,y)> k+ 2\e+2$.
Then the segment $[x,y]$ contains a subsegment $[u,v]$, where $u,v$ are vertices of $\T$, such that $\d _\T (\{ u,v\}, \{ x,y\}) \ge \e$ and $\d (u,v)> k$.
By our assumption, $|\pst_G(\{ u,v\})|=1$ and hence $|\pst _G^\e(\{ x,y\}) | < 2(2\e+1)$ by Lemma \ref{lem:epst}. Thus the action of $G$ on $\T$ is acylindrical by Remark \ref{rem:equiv_def-acyl}.
\end{proof}

Recall that every acylindrically hyperbolic group $G$ contains a unique maximal finite normal subgroup $K(G)$, called the \emph{finite radical} of $G$ (see Theorem \ref{thm:prop-ah-2}).

\begin{lemma}\label{lem:at}
Let $G$ be the a group acting acylindrically in the sense of Sela and without inversions on a tree $\T$. Then either $G$ is acylindrically hyperbolic with trivial finite radical,
or it is infinite cyclic or infinite dihedral, or $G$ fixes a vertex of $\T$.
\end{lemma}

\begin{proof}
By Lemma \ref{lem:Sa} the action of $G$ on $\T$ is acylindrical (as defined in the Introduction). So, in view of Theorems \ref{class} and \ref{acyl},
$G$ is either acylindrically hyperbolic, or virtually cyclic containing a hyperbolic element,
or has bounded orbits. In the latter case $G$ fixes a vertex $x$ of $\T$ (see \cite[I.4.7]{D-D}). %Since $G$ acts without inversions, $x$ must be a vertex of $\T$.

Now assume that $G$ is acylindrically hyperbolic or virtually cyclic and contains a hyperbolic element $h$.
Recall that every virtually cyclic group $E$ has a finite normal subgroup $K\leqslant E$ such that $E/K$ is either infinite cyclic or infinite dihedral.
Thus to prove the lemma it suffices to show that $G$ has no non-trivial finite normal subgroups.

Suppose $K$ is a finite normal subgroup of $G$. After passing from $h$ to some non-trivial power, we can assume that $h$ belongs to the centralizer $\C_G(K)$. Since $K$ is finite, it fixes a vertex $v$ of $T$.
Assume that the action of $G$ on $\T$ is $k$-acylindrical in the sense of Sela. Since $h$ is hyperbolic, there exists $n\in \mathbb N$ such that $\d (u,v)\ge k$ for $u=h^nv$. Obviously for every $g\in K$, we have $gu=gh^nv=h^ngv=h^nv=u$; hence $K$ fixes $u$ as well.
Now the $k$-acylindricity of the action in the sense of Sela implies that $K=\{ 1\}$.
\end{proof}

The next lemma is proved in \cite{DGO}.

\begin{lemma}[{\cite[Corollary 6.13]{DGO}}]\label{lem:613}
Suppose  that $H$ is a proper infinite
hyperbolically embedded subgroup of a group $G$. Let $K$ be a subgroup of $G$ such that $K\cap H$ is infinite and $K\setminus H\ne \emptyset$. Then $K$ contains  proper infinite hyperbolically embedded subgroups.
\end{lemma}

Finally, we will need the following result, which is a (probably well-known) consequence of the projective plane theorem proved in \cite{Eps}. By $F_n$ we denote the free group of rank $n$.

\begin{lemma}\label{epstein}
$F_2\times \mathbb Z_2\notin \MM$.
\end{lemma}

\begin{proof}
Assume that $F_2\times \mathbb Z_2\in \MM$. Then $F_2\times \mathbb Z_2=\pi_1(M)$ for some $3$-manifold $M$. By \cite[Theorem 9.5]{Eps}, if $\mathbb Z\times \mathbb Z_2$ embeds in the fundamental group of some $3$-manifold $N$, then either $\pi_1(N)\cong \mathbb Z\times \mathbb Z_2$ or $\pi_1(N)$ is a non-trivial free product. Applying this result to $M$ we obtain that $F_2\times \mathbb Z_2$ is a non-trivial free product. However, a non-trivial free product must have trivial center. A contradiction.
\end{proof}

Now we are ready to prove the main result of this section, which is slightly stronger than Theorem \ref{main-3-dim} from  Section \ref{sec:results}  (see parts (I$^\prime$) and (II$^\prime$)).

\begin{thm}\label{main-3d}
Let $M$ be a compact orientable $3$-manifold and let $G$ be a subgroup of $\pi_1(M)$. Then exactly one of the following three conditions holds.
\begin{enumerate}
\item[(I$^\prime$)] $G$ is acylindrically hyperbolic with trivial finite radical.

\item[(II$^\prime$)] $G$ contains an infinite cyclic normal subgroup $Z$ and $G/Z$ is acylindrically hyperbolic. In fact, $G/Z$ is virtually a subgroup of a surface group in this case.

\item[(III)] $G$ is virtually polycyclic.
\end{enumerate}
\end{thm}

\begin{proof}[Proof of Theorem \ref{main-3-dim}]
Note that the three conditions (I$^\prime$), (II$^\prime$) and (III) are mutually exclusive by Theorem \ref{thm:elem-prop-ah} (c) and Example \ref{ex:non-ah} (a). Thus we only need to show that at least one of them holds.

If the boundary $\partial M$ is non-empty, we take the double $N=M\cup_{\partial M} M$. Since $M$ is a retract of $N$, $\pi_1(M)$ embeds in $\pi_1(N)$ (cf. \cite[Lemma 1.6]{AFW}). %Thus without loss of generality we can assume that $M$ is also closed.

%Further if $M$ is non-orientable, we pass to its orientable double cover $M_0$ and consider the index $2$ subgroup $G_0\leqslant G$ such that $G_0= \pi_1(M_0)$.
%Note that if one of the  conditions (I$^\prime$), (II$^\prime$) or (III) holds for $G_0$, then the same condition holds for $G$. Indeed, for (III) this is obvious.
%Further assume that $G_0$ satisfies (II$^\prime$). Thus $G_0$ contains a finite index subgroup $N \lhd G_0$ such that there is an infinite cyclic subgroup
%$Z \lhd N$ and $N/Z$ is isomorphic to a subgroup of the fundamental group of a compact surface. After replacing $N$ with a finite index subgroup if necessary we can suppose that $N \lhd G$ and $Z$ is central in $N$. Since $N/Z$ cannot be cyclic (as (III) does not hold for $G_0$), the group $N/Z$ is either a non-abelian free group or a closed hyperbolic surface group. In particular, $N/Z$ has trivial center, which means that $Z$ is the center of $N$. Therefore $Z$ is characteristic in $N$ and so $Z \lhd G$. The group $N/Z$ is acylindrically hyperbolic and has finite index in the quotient $G/Z$. By Lemma~\ref{lem: com} $G/Z$ is also acylindrically hyperbolic  and hence (II$^\prime$) holds for $G$. Finally assume that (I$^\prime$) holds for $G_0$. Again, by Lemma \ref{lem: com}, $G$ is acylindrically hyperbolic. If $G$ has non-trivial finite radical, then $G\cong G_0\times \mathbb Z_2$ as $|G:G_0|=2$.
%Therefore, by Theorem \ref{thm:elem-prop-ah}.(e), $G$ contains $F_2\times \mathbb Z_2$, which contradicts Lemma \ref{epstein}. Hence $G$ satisfies (I$^\prime$).

Thus we can (and will) assume that $G$ is a subgroup of $\pi_1(M)$ for an orientable closed $3$-manifold $M$. By the Kneser-Milnor theorem, $M$ decomposes as a connected sum $M_1\# \cdots \# M_n$ of prime $3$-manifolds. Consequently, $\pi _1(M)=\pi_1(M_1)\ast \cdots \ast \pi_1(M_n)$.
Suppose, first, that $G$ is not conjugate to a subgroup of $\pi_1(M_i)$ for any $i$. Then by Kurosh theorem $G$ is a free product of two non-trivial groups and whence (I$^\prime$) or (III) holds in this case
by Lemma \ref{lem:at}, because the action of this free product on the associated Bass-Serre tree is $1$-acylindrical in the sense of Sela and has no global fixed points.

Now, assume that $G$ is conjugate to a subgroup of some $\pi_1(M_i)$.  Since $M$ is closed and orientable, so is $M_i$.
Recall that every orientable prime manifold, other than $\mathbb S^2\times \mathbb S^1$, is irreducible (cf. \cite[Lemma 3.13]{Hem}). Since
$\pi_1(\mathbb S^2\times \mathbb S^1)$ is cyclic, it suffices to consider the case when $G$ is isomorphic to a subgroup of a closed orientable irreducible $3$-manifold.

From now on, we assume that $G\leqslant \pi_1(M)$ for a closed orientable irreducible $3$-manifold $M$. Recall that if the fundamental group of such a manifold $M$ is infinite, then it is torsion-free
(see, for example, \cite[(C.2)]{AFW}). Thus we can also assume that $\pi_1(M)$ is torsion-free. In particular, in (I$^\prime$) it will suffice to prove acylindrical hyperbolicity; triviality of the finite radical will follow from the absence of
torsion.

By Lemma \ref{WZ} either $M$ has a finite-sheeted covering space that is a torus bundle over a circle or the action of $\pi_1(M)$ on the Bass-Serre tree $\T$ associated to the JSJ
decomposition of $\pi_1(M)$ is acylindrical in the sense of Sela. In the former case (III) holds. In the latter case, by Lemma \ref{lem:at}, $G$ either satisfies one of the conditions
(I$^\prime$), (III), or fixes a vertex of $\T $. Thus it remains to consider the case when $G$ is a subgroup of the fundamental group of a Seifert fibered piece or an atoroidal piece of the JSJ decomposition.

Suppose, first, that $G$ is a subgroup of the fundamental group of a Seifert fibered space $S$.
If $\pi _1(S)$ is finite, we are done. If $\pi _1(S)$ is infinite, then it  contains an infinite normal cyclic subgroup $Z$ such that the quotient $H:= \pi_1(S)/Z$ contains a surface subgroup of finite index (see \cite[Sec. 1.6]{M_Kapovich}).
In particular, $H$ is either virtually $\mathbb Z^k$ for $k\in \{ 0,1,2\}$ or is a non-elementary hyperbolic group.  If $H$ is
virtually $\mathbb Z^k$ for $k\in \{ 0,1,2\}$, so is $G/I$, where $I:= G\cap Z$. If $H$ is non-elementary hyperbolic, we fix any finite generating set $X$ of $H$ and note that the
natural action of $G/I$ on the Cayley graph  $\Gamma (H,X)$ is acylindrical since this is true for the action of $H$. By Theorem \ref{class}, the action of $G/I$ on
$\Gamma (H,X)$ is elliptic (in which case $G/I$ is finite as the action is proper), or $G/I$ is virtually cyclic, or $G/I$ is acylindrically hyperbolic.
Summarizing, we conclude that, in any case, $G/I$ is either virtually $\mathbb Z^k$ for $k\in \{ 0,1,2\}$ or acylindrically hyperbolic.
Thus if $|I|=\infty$, $G$ satisfies (II$^\prime$) or (III). Otherwise, $I=\{ 1\}$,  $G\cong G/I$, and again (I$^\prime$) or (III) holds.

Now assume that $G$ fixes a vertex of $\T $ corresponding to an atoroidal piece. As we explained above, in this case $G$ is a subgroup of a group $R$ hyperbolic relative to
free abelian subgroups of rank $2$. Let $A_1, \ldots , A_m$ be the collection of peripheral subgroups of $R$, $X$ a finite generating set of $R$, $\mathcal A=\bigsqcup_{i=1}^m A_i$.
Then every $A_i$ is hyperbolically embedded in $R$ (see \cite[Prop. 4.28 and Remark 4.26]{DGO}).

If $G\ne \{ 1\}$, fix any $g\in G\setminus \{ 1\}$. Since $G$ is torsion-free, by
\cite[Cor. 4.20]{Osi06}
the element $g$ is either conjugate to an element of $A_i$ for
some $i$, or acts hyperbolically on $\Gamma (R, X\sqcup \mathcal A)$ (note that the term ``hyperbolic element" is used in a slightly different sense in \cite{Osi06}). Again, we consider two cases.

If $g$ is conjugate to an element of $A_i$ for some $i$, then passing from $G$ to its conjugate if necessary, we can assume that $G\cap A_i$ is infinite. If $G\leqslant A_i$,
then we have (III). Otherwise $G$ contains a proper infinite hyperbolically embedded subgroup by Lemma \ref{lem:613}, and hence
$G$ is acylindrically hyperbolic by Theorem \ref{acyl}.

Finally assume that $g$ is hyperbolic with respect to the action on  $\Gamma (R, X\sqcup \mathcal A)$. In particular, the action of $G$ is not elliptic. By \cite[Prop. 5.2]{Osi13}, the action of $R$ (and, hence, of $G$)
on the Cayley graph $\Gamma (R, X\sqcup \mathcal A)$ is acylindrical. Applying now Theorem \ref{class} we conclude that $G$ is either virtually cyclic or acylindrically hyperbolic. This completes the proof of the theorem.
\end{proof}

\begin{proof}[Proof of Corollary \ref{cor:geom}] Let $M$ be a compact irreducible orientable $3$-manifold and $G:=\pi_1(M)$. 
Suppose that $G$ is neither virtually polycyclic nor acylindrically hyperbolic.
Then, by Theorem~\ref{main-3-dim}, $G$ contains an infinite cyclic normal subgroup $Z$. Hence, since $M$ is irreducible, we can use \cite[Thm. 1.39]{M_Kapovich} claiming that
$M$ is a Seifert fibered manifold.
\end{proof}

\section{Graph products}\label{sec:GP}

\subsection{Parabolic subgroups}\label{sec:par_sbgps}
{Suppose that $\ga$ is a \emph{finite} simplicial graph with vertex set $V=V\ga$,
and $\G=\{G_v \mid v \in V\ga\}$ is a family of \emph{non-trivial} groups (called \textit{vertex groups}). %Let $G=\ga\G$ be the corresponding graph product.
The \emph{graph product} $G=\ga \G$, of this collection of groups with respect to $\ga$,
is the group obtained from the free product of the $G_v$, $v \in V\ga$, by adding the relations
$$[g_v, g_u]=1  \text{ for all }  g_v\in G_v,\, g_u\in G_u \text{ whenever $\{v,u\}$ is an edge of } \ga.$$

Graph product of groups is a natural group-theoretic construction generalizing free products (when $\ga$ has no edges) and direct products (when $\ga$ is a complete graph) of
groups $G_v$, $v \in V\ga$. Graph products were first introduced and studied by E. Green in her Ph.D. thesis \cite{Green}. Further properties of graph products have been investigated in the works of
Hermiller and Meier \cite{H-M}, Hsu and Wise \cite{HsuWise}, Januszkiewicz and \'Swi\c{a}tkowski \cite{J-S}, and Antol\'{i}n and Minasyan \cite{A-M-Tits}.
}

Recall that the \emph{link} $\link_{\Gamma}(v)$, of a vertex $v \in V\Gamma$, is the set of all vertices of $\ga$ that are adjacent to $v$ (excluding $v$ itself); in other words,
$\link_{\Gamma}(v)\coloneqq \{u\in V\Gamma \mid \{v,u\}\in E\Gamma \}$. For a non-empty subset $A \subseteq V\Gamma$, we define $\link_\ga(A)\coloneqq \bigcap_{v \in A} \link_\ga(v)$.
Note that $A \cap \link_\ga(A)=\emptyset$.

For any subset $A\subseteq V\Gamma,$ the subgroup $G_A\leqslant G$, generated by $\{G_v \mid v\in A\}$, is called \textit{full}; according to a standard convention,
$G_\emptyset=\{1\}$.  By the Normal Form Theorem {for graph products (see \cite{Green} or \cite{HsuWise})},
$G_A$ is the graph product of the family $\G_A:=\{G_v \mid v \in A\}$ with respect to the full subgraph $\Gamma_A$ of $\ga$,
spanned on the vertices from $A$. It is easy to see that there is a canonical retraction $\rho_A: G\to G_A$,
defined (on the generators of $G$) by $\rho_A(g) := g$ for each $g \in G_v$ with $v \in A$, and
$\rho_A(h) := 1$ for each $h \in G_u$ with $u \in V\setminus A$.

Any conjugate of a full subgroup in $G$ is called {\it parabolic}. Parabolic subgroups of graph products were studied and used by Antol\'in and the first author in \cite{A-M-Tits}.
In this subsection we summarize some of the known properties of parabolic subgroups of graph products.

%Throughout this section, $\ga$ will be a finite graph,  $V:=V\ga$ will denote the vertex set of $\Gamma$ and $G$ will denote the graph product
%of a family of non-trivial groups $\{G_v \mid v \in V\}$ with respect to $\ga$.
For a subset $X \subseteq G$, $\C_G(X)$ and $\Nn_G(X)$ will denote the centralizer and the normalizer of $X$ in $G$ respectively.
All of the claims in the lemma below were proved in Section 3 of \cite{A-M-Tits}.

\begin{lemma}\label{lem:parab_props} %fffdsgzdsghasdfhfjfz
\begin{itemize}
\item[]
\item[(i)] If $S,T \subseteq V$ and $g_1,g_2 \in G$ are such that $g_1G_Sg_1^{-1} \subseteq  g_2G_Tg_2^{-1}$ then $S \subseteq T$.
\item[(ii)] If $P$ is a parabolic subgroup of  $G$ and $gPg^{-1} \subseteq P$ for some $g \in G$, then $gPg^{-1}=P$.
\item[(iii)] If $P=fG_Sf^{-1}$, where $S \subseteq V$ and $f \in G$, then
$\Nn_G(P)=f G_{S \cup \link_\ga(S)}f^{-1}$ is a parabolic subgroup of $G$.
\item[(iv)] The intersection of two parabolic subgroups of $G$ is also a parabolic subgroup of $G$.
\end{itemize}
\end{lemma}

\begin{df} Suppose that $P=fG_Sf^{-1}$ is a parabolic subgroup of $G$, for some $S \subseteq V$ and $f \in G$. We define the \emph{parabolic dimension} $\pdim_\ga(P)$ of $P$ to be
the number of elements in $S$.
\end{df}

Part (i) of Lemma \ref{lem:parab_props} shows that the parabolic dimension is well-defined as $S$ is completely determined by $P$.

\begin{lemma}\label{lem:pdim} Let $P$ and $Q$ be parabolic subgroups of $G$.
\begin{itemize}
\item[(a)] If $P \subseteq  Q$ then $\pdim_\ga(P) \le \pdim_\ga(Q)$ with equality if and only if $P=Q$.
\item[(b)] If $ \pdim_\ga(P)\le \pdim_\ga(P \cap Q)$ then $P \subseteq Q$.
\end{itemize}
\end{lemma}

\begin{proof} First note that, according to Lemma \ref{lem:parab_props}.(iv), $P \cap Q$ is also a parabolic subgroup of $G$. Therefore $\pdim_\ga(P \cap Q)$ is defined and
part (b) is an immediate consequence of part (a).

To prove part (a), suppose that $P \subseteq Q$ and $P=g_1G_Sg_1^{-1}$, $Q= g_2G_Tg_2^{-1}$ for some $S,T \subseteq V$, $g_1,g_2 \in G$.
Lemma \ref{lem:parab_props}.(i) implies that $S \subseteq T$, hence $\pdim_\ga(P)=|S|\le |T|=\pdim_\ga(Q)$. Now, if $\pdim_\ga(P) = \pdim_\ga(Q)$
then $S = T$ and thus $Q=g P g^{-1}$, where $g:=g_2g_1^{-1} \in G$. By Lemma \ref{lem:parab_props}.(ii), the inclusion of $P$ in $Q=gPg^{-1}$ yields that $P=gPg^{-1}=Q$, as required.
\end{proof}

By part (iv) of Lemma \ref{lem:parab_props} the intersection of finitely many parabolic subgroups is again a parabolic subgroup.
In fact a more general statement holds:

\begin{lemma}\label{lem:inter_parab_fam} Intersection of any family of parabolic subgroups of $G$ is equal to the intersection of a finite subfamily.
In particular, this intersection is itself a parabolic subgroup of $G$.
\end{lemma}

\begin{proof} Consider a family $\{Q_i \mid i \in I\}$ of parabolic subgroups of $G$. For any finite subset $J$ of $I$, the intersection
$P_J:=\bigcap_{j \in J} Q_j$ is a parabolic subgroup
of $G$, thus its parabolic dimension is defined. Since this dimension is a non-negative integer,  and non-negative integers are well-ordered,
the function, which assigns to each finite subset $J \subseteq I$ the value $\pdim_\ga(P_J)$, reaches its minimum on
some finite subset $J_1 \subseteq I$.
Consider any $i \in I$ and set $J':=J_1 \cup \{i\}$. Then $\pdim_\ga(P_{J_1}) \le \pdim_\ga(P_{J'})=\pdim_\ga(P_{J_1} \cap Q_i)$ by the choice of $J_1$,
hence Lemma \ref{lem:pdim}.(b)
implies that $P_{J_1} \subseteq Q_i$. Since the latter is true for all $i \in I$ we can conclude that
$\bigcap_{i \in I} Q_i= P_{J_1}= \bigcap_{j \in {J_1}} Q_j$.
\end{proof}

In \cite[Prop. 3.10]{A-M-Tits} it was proved that any subset $X \subseteq G$ is contained in a unique minimal
parabolic subgroup $\pc_\ga(X)$, which is called \emph{the parabolic closure} of $X$. The next lemma lists two important properties of parabolic closures, which were
also established in \cite[Sec.~3]{A-M-Tits}.

\begin{lemma}\label{lem:parab_clos} Let $X \subseteq G$ be an arbitrary subset and let $P=\pc_\ga(X)$ be its parabolic closure in the graph product $G$.
\begin{itemize}
\item[(i)] The normalizer $\Nn_G(X)$ is contained in the normalizer $\Nn_G(P)$.
\item[(ii)] There exists a finite subset $X' \subseteq X$ such that $\pc_\ga(X)=\pc_\ga(X')$.
\end{itemize}
\end{lemma}

\begin{lemma}\label{lem:par_clos_in_spec} Consider any subset $A \subseteq V$ and the  corresponding full subgroup $G_A$ of $G$. As we know, $G_A$ is
equal to the graph product $\ga_A\G_A$, where $\ga_A$ is the full subgraph of $\ga$ spanned by $A$, and $\G_A:=\{G_v \mid v \in A\}$.
Then for an arbitrary subset $X \subseteq G_A$, the parabolic closure $\pc_{\ga_A}(X)$ of $X$ in $G_A=\ga_A\G_A$ coincides with the parabolic closure
$\pc_\ga(X)$ of $X$ in $G=\ga\G$.
\end{lemma}

\begin{proof} Indeed, suppose that $\pc_\ga(X)=fG_Tf^{-1}$ for some $T \subseteq V$ and $f \in G$.
Since $X \subseteq G_A$ and $G_A$ is a parabolic subgroup of $G$,
we have $fG_Tf^{-1} \subseteq G_A$, implying that $T \subseteq A$ by Lemma \ref{lem:parab_props}.(i).
Let $\rho_A:G \to G_A$ be the canonical retraction. The inclusions $fG_Tf^{-1},G_T \subseteq G_A$, yield that
$fG_Tf^{-1} =\rho_A(fG_Tf^{-1})=gG_Tg^{-1}$, where $g:=\rho_A(f)\in G_A$. Hence $\pc_\ga(X)=gG_Tg^{-1}$ is a parabolic subgroup of $G_A$ containing $X$.
Consequently $\pc_{\ga_A}(X)\subseteq \pc_\ga(X)$. On the other hand, $\pc_{\ga}(X)\subseteq \pc_{\ga_A}(X)$ because any parabolic subgroup of $G_A=\ga_A\G_A$ is also
a parabolic subgroup of $G=\ga\G$. Therefore $\pc_{\ga_A}(X)= \pc_\ga(X)$.
\end{proof}

\begin{df}\label{df:join} Let $K=fG_Sf^{-1}$ be a parabolic subgroup of $G$, where $S \subseteq V$ and $f \in G$. We shall say that $K$ is a {\it join subgroup} of $G$
if there are non-empty disjoint subsets $A,B \subset V$ such that $S=A \sqcup B$ and $B \subseteq \link_\ga(A)$. If, in addition, both $G_A$ and $G_B$ are infinite
then we will say that $K$ is a {\it principal join subgroup} of $G$.
\end{df}

In other words, a join subgroup is a parabolic subgroup which naturally splits as a direct product of smaller parabolic subgroups, and in a principal join subgroup these smaller
parabolic subgroups must be infinite.

Using the terminology from \cite{A-M-Tits}, the {\it essential support} $\esupp_\ga(X)$ of a subset $X \subseteq G$ is defined as the smallest subset $S$ of $V$ such that
$X$ is contained in a conjugate of $G_S$ in $G$. Equivalently, $S=\esupp_\ga(X)$ if $\pc_\ga(X)=f G_S f^{-1}$ for some $f \in G$. A subset $T$ of $V=V\ga$ is said to be \emph{irreducible} if
$T$ cannot be represented as a union of disjoint non-empty subsets $A,B \subset S$ such that $B \subseteq \link_\ga(A)$.
Thus $T$ is irreducible if and only if $\ga_T$ is an irreducible graph (as defined in Subsection \ref{subsec:gp}), which is also equivalent to $G_T$ not being a join subgroup of $G$.

The above definitions together with Lemma \ref{lem:parab_props}.(i) immediately imply the following fact:

\begin{rem}\label{rem:cont_in_join} Suppose that $X$ is a subset of the graph product $G$. Then $X$ is contained in a join subgroup of $G$ if and only if
there are non-empty disjoint subsets $A,B \subset V$ such that $B \subseteq \link_\ga(A)$ and $\esupp_\ga(X) \subseteq A \sqcup B$. Similarly, $X$ is
contained in a principal join subgroup of $G$ if and only if there exist non-empty disjoint subsets $A, B \subset V$ such that $B \subseteq \link_\ga(A)$,
$\esupp_\ga(X) \subseteq A \sqcup B$ and $|G_A|=\infty$, $|G_B|=\infty$.
\end{rem}

\begin{lemma}\label{lem:H_in_join}  Let $H$ and $N$ be non-trivial subgroups of $G$ such that $H$ normalizes $N$ and $H$ is not contained in the parabolic closure $P:=\pc_\ga(N)$.
Then $H$ is contained in a join subgroup of $G$. Moreover, if $|N|=\infty$ and $|H:H\cap P|=\infty$ then $H$ is contained in a principal join subgroup of $G$.
\end{lemma}

\begin{proof} By definition, $P=fG_Af^{-1}$ for some non-empty subset $A \subset V$ and an element $f \in G$. By Lemma \ref{lem:parab_clos}.(i), since $P$ is the parabolic closure of $N$,
$H \leqslant \Nn_G(N) \leqslant \Nn_G(P)$ in $G$. And, according to Lemma \ref{lem:parab_props}.(iii), $\Nn_G(P)=f G_{A \cup \link_\ga(A)}f^{-1}$. Since
$H$ is contained in $\Nn_G(P)$ but not in $P$, it follows that $B:=\link_\ga(A) \neq \emptyset$. Hence $\Nn_G(P)$ is a join subgroup of $G$ containing $H$ and the first claim is proved.

Now, if $N$ is infinite then $G_A$ is infinite. And if $|H:H\cap P|=\infty$ then $|\Nn_G(P):P|=\infty$, and so $|G_{A\cup B}:G_A|=\infty$. Recalling that $G_A$ commutes with $G_B$
(by definition of $B$), we see that $G_{A\cup B}=G_A G_B$, hence $|G_B|=|G_{A\cup B}:G_A|=\infty$. Thus $\Nn_G(P)$ is a principal join subgroup of $G$ containing $H$, and
the second claim of the lemma also holds.
\end{proof}

\subsection{Application of Bass-Serre theory to graph products}\label{sec:wpd_in_gp}
As before, we fix a finite simplicial graph $\ga$ with the vertex set $V$ and consider the graph product $G=\ga\G$, of a family of non-trivial
groups $\G=\{ G_v \mid v \in V\} $, with respect to $\ga$.

The definition of graph product immediately implies that for any $v \in V$ the group $G$ naturally splits as the
free amalgamated product: $G=G_A*_{G_C}G_B$, where $C=\link_\ga(v)$, $B=\{v\}\cup\link_\ga(v)$ and $A=V-\{v\}$
(cf. \cite[Lemma 3.20]{Green}). It follows that $G$ acts (simplicially by isometries and without edge inversions) on the Bass-Serre tree $\T$ associated to this splitting,
and full vertex and edge stabilizers for this action are parabolic subgroups of $G$ (conjugates of $G_A$, $G_B$ and $G_C$).

\begin{lemma}\label{lem:H_has_hyp_elt} Consider a subgroup $H\leqslant G$ and a vertex $v \in \esupp_\ga(H)$.
Suppose that $\esupp_\ga(H) \not\subset \{v\} \cup \link_\ga(v)$ in $\ga$.
Let $A:=V \setminus \{v\}$, $B:=\{v\} \cup \link_\ga(v)$ and $C:=\link_\ga(v)$. As we know, $G$ splits as an amalgamated product
$G=G_A *_{G_C} G_B$ and acts on the corresponding simplicial Bass-Serre tree $\T$. Then there exists an element $h \in H$ which acts as a hyperbolic isometry of $\T$.
\end{lemma}

\begin{proof} Arguing by contradiction, suppose that $H$ contains no hyperbolic elements; then every element of $H$ is elliptic.
By Lemma \ref{lem:parab_clos}.(ii), there is a finite subset $X' \subseteq H$ such that $\pc_\ga(H)= \pc_\ga(X')$.
The subgroup $H'\coloneqq \gen{X'} \leqslant H$ is finitely generated and every element of $H'$ fixes a vertex of $\mathcal T$, therefore
$H'$ fixes some vertex of $\mathcal T$ (see \cite[I.6.5, Cor. 3]{Serre}). But full $G$-stabilizers of vertices of $\mathcal T$ are conjugates of $G_A$ and $G_B$ in $G$,
which are parabolic subgroups of $G$.  Therefore there is $f \in G$ such that $H' \leqslant fG_Df^{-1}$ where $D \subset V$ is either $A$ or $B$.
It follows that $$\pc_\ga(H)=\pc_\ga(X') =\pc_\ga(H') \leqslant fG_Df^{-1}.$$
One can now apply Lemma \ref{lem:parab_props}.(i) to conclude that $\esupp_\ga(H) \subseteq D$, and so $D \neq A$ by the definition of $A$. Therefore $D=B$ and
$\esupp_\ga(H) \subseteq B$. Recall that $B=\{v\} \cup \link_\ga(v)$, thus $\esupp_\ga(H)\subseteq \{v\} \cup \link_\ga(v)$,
which contradicts our assumptions.
\end{proof}

Suppose that a group $G$ acts on a tree $\T$  and let $\d_\T(\cdot,\cdot)$ be the metric on $\T$.
For any two points $x, y $ of $\T$, $[x,y]$ will denote the geodesic segment between $x$ and $y$ in $\T$.
If $\mathcal{Y}$ is a subset of $\T$, we will write $\st_G(\mathcal{Y})$ and $\pst_G(\mathcal{Y})$ to denote the setwise and pointwise stabilizers of
$\mathcal Y$ in $G$ respectively. Evidently $\pst_G(\mathcal{Y})\lhd \st_G(\mathcal{Y})$.

\begin{lemma}\label{lem:axis_pst} Consider an element $h \in G$ that acts on the tree $\T$ as a hyperbolic isometry, and let $\mathcal{L} \subseteq \T$ be the axis of $h$.
Assume that $\pst_G(\L)=\pst_G([x,y])$ for some geodesic subsegment $[x,y]$ of $\L$ with $x,y \in \L$. Then there exists $M> 0$ such that for any points
$z,w \in \L$ with $\d_\T(z,w) \ge M$ one has $\pst_G([z,w])=\pst_G(\L)$.
\end{lemma}

\begin{proof} Recall that $h$ acts on $\L$ by translation with positive translation length $\|h\|>0$.
Let $M:=\d_\T(x,y) + \|h\|> 0$, and consider any subsegment $[z,w]$ of $\L$ with $\d_\T(z,w)\ge M$.
Then for some integer $m \in \Z$, $h^m  [x,y]$ will be contained in $[z,w]$.
Consequently one has $$\pst_G([z,w])\subseteq \pst_G(h^m [x,y])=h^m \pst_G([x,y]) h^{-m} = h^m \pst_G(\L)h^{-m}. $$
Clearly $h \in \st_G(\L)$ normalizes $\pst_G(\L)$, thus $\pst_G([z,w])\subseteq \pst_G(\L)$.
On the other hand, since $[z,w]$ is a subsegment of $\L$ one sees that  $\pst_G(\L) \subseteq \pst_G([z,w])$. Therefore $\pst_G([z,w])=\pst_G(\L)$, as claimed.
\end{proof}

For the remainder of this section we assume that $G=\ga\G$,  $H\leqslant G$ is a subgroup and $v \in \esupp_\ga(H)$ is
such that $\esupp_\ga(H) \not\subset \{v\} \cup \link_\ga(v)$ in $\ga$.
As before, set $A:=V \setminus \{v\}$, $B:=\{v\} \cup \link_\ga(v)$ and $C:=\link_\ga(v)$. Let $\T$ be the Bass-Serre tree associated to the natural splitting
$G=G_A *_{G_C} G_B$.

\begin{lemma}\label{lem:pst-parab} If $\L$ is the axis of some hyperbolic element $h \in G$ then $\pst_G(\L)$ is a parabolic subgroup of $G$, which is
contained in a conjugate of $G_C$. Moreover, there exist vertices $x,y$ on $\L$ such that $\pst_G(\L)=\pst_G([x,y])$.
\end{lemma}

\begin{proof} By definition, $\pst_G(\L)=\bigcap_{e \in \mathcal{E}} \pst_G(e)$, where $\mathcal{E}$ is the set of all edges of $\L$.
Since each $\pst_G(e)$ is a conjugate of $G_C$ in $G$,
$\pst_G(L)$ is a parabolic subgroup of $G$ by Lemma \ref{lem:inter_parab_fam}. In fact Lemma \ref{lem:inter_parab_fam} implies that
$\pst_G(\L)=\bigcap_{e \in \mathcal{E}'} \pst_G(e)=\pst_G(\mathcal{E}')$
for some finite collection $\mathcal{E}'$ of edges of $\L$. Since any finite collection of edges of $\L$
is contained in a finite subsegment we can find vertices $x,y \in \L$
such that $\pst_G(\L)=\pst_G(\mathcal{E}')=\pst_G([x,y])$.
\end{proof}

With the above assumptions on $H$, Lemma \ref{lem:H_has_hyp_elt} tells us that $H$ contains at least one element that acts as a hyperbolic isometry on $\T$.
The main technical statement of this section is the following proposition:
%In view of Lemma \ref{lem:pst-parab} one can find

\begin{prop}\label{prop:pst(axis)-normal} Let $h \in H$ be a hyperbolic element such that $\pdim_\ga(\pst_G(\L))$ is minimal, where $\L$ is the axis of $h$.
Then the parabolic subgroup $\pst_G(\L)$ is normalized by $H$ in $G$.
\end{prop}

\begin{proof} We will first show that for any element $g \in H$, which also acts as a hyperbolic isometry on $\T$, $\pst_G(\L) \subseteq \pst_G(\M)$, where $\M$ is the axis of $g$.
Let $M_1>0$ and $M_2>0$ be the constants obtained by applying Lemma \ref{lem:axis_pst} to $h$ and $g$ respectively (this can be done in view of Lemma \ref{lem:pst-parab}).
Choose a natural number $n \in \N$ so that $n\|h\| \ge M_1$ and $n\|g\| \ge M_2$. We need to consider two separate cases.

\emph{Case 1.} The intersection of $\L$ with $\M$  contains at least one edge $e$. After replacing $h$ and $g$ with their inverses, if necessary,
we can suppose that both  $h$ and $g$ translate $e$ (cf. Def.~\ref{df:edge_transl}).
Set $e':=h^{-n}  e$ and note that the element $f:=g^nh^n \in H$ translates $e'$.
Therefore $f$ is hyperbolic and $[e'_-,(f  e')_+]$ lies on the axis $\mathcal{N}$ of $f$. By construction,
the edge $e$ is contained in the segment $[e'_-,(f  e')_+]$ (see Figure \ref{fig:1}).

\begin{figure}[ht]
\input{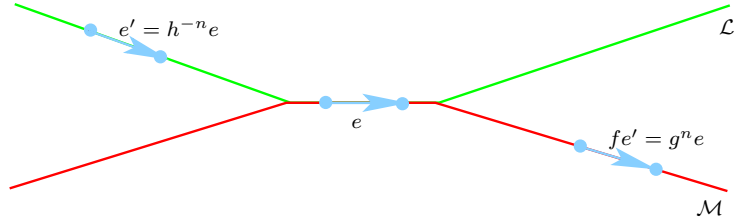}
\caption{Illustration of Case 1. }\label{fig:1}
\end{figure}

Observe that $e'$ is an edge of $\L$ (as $\L$ is $\langle h \rangle$-invariant and $e \in \L$), hence the intersection of $\mathcal{N}$ and $\L$ contains the segment
$[e'_-,e_-]$, whose length is equal to $n \|h\|$, which is greater than $M_1$. Therefore Lemma \ref{lem:axis_pst} implies that
\begin{equation} \label{eq:N<L} \pst_G(\mathcal{N})\subseteq \pst_G([e'_-,e_-])=\pst_G(\L).
\end{equation}

Similarly, as $g^n  e=f  e'$, we have $[e_-,(g^n  e)_-] \subseteq \cN \cap \M$, and so
\begin{equation} \label{eq:N<M} \pst_G(\mathcal{N})\subseteq \pst_G([e_-,(g^n  e)_-])=\pst_G(\M). \end{equation}
However, since $\L$ was chosen so that $\pdim_\ga(\pst_G(\L))$ is minimal, we know that $$\pdim_\ga(\pst_G(\L)) \le \pdim_\ga(\pst_G(\cN)).$$ Combining \eqref{eq:N<L}
with Lemma \ref{lem:pdim}.(a) we get that $\pst_G(\L) = \pst_G(\cN)$, and hence \eqref{eq:N<M} yields that $\pst_G(\L) \subseteq \pst_G(\M)$.

\emph{Case 2.} The intersection of $\L$ with $\M$ is empty or consists of a single vertex. Let $p \in \L$ and $q \in \M$ be vertices with the property
$\d_\T(p,q)=\d_\T(\L,\M)$; thus  (in the terminology of \cite{CullerMorgan}) $[p,q]$ is the \emph{spanning geodesic} between $\L$ and $\M$. In \cite[1.5]{CullerMorgan}
it is shown that the element $f:=g^n h^n \in H$ is hyperbolic and the segment $[h^{-n}  p, g^n  q]$ lies on the axis $\cN$ of $f$.
By construction, both $p$ and $q$ belong to this segment (see Figure \ref{fig:2}).

\begin{figure}[ht]
\input{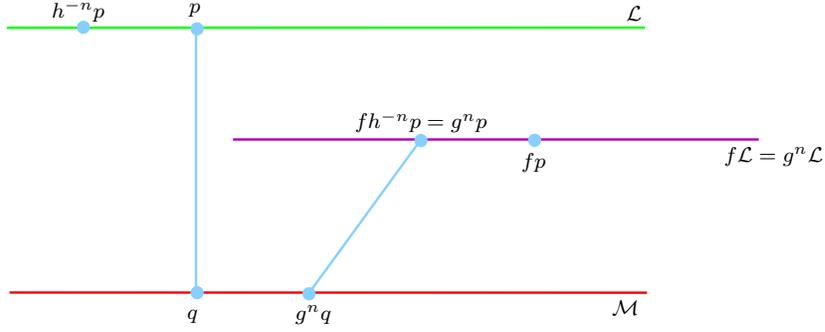}
\caption{Illustration of Case 2. }\label{fig:2}
\end{figure}

Again, we observe that $[h^{-n}  p,p] \subseteq \L \cap \cN$ and $\d_\T(h^{-n}  p,p)=n \|h\|\ge M_1$. Therefore Lemma \ref{lem:axis_pst} tells us that
$\pst_G(\mathcal{N})\subseteq \pst_G([h^{-n}  p,p])=\pst_G(\L)$. Similarly, $\pst_G(\mathcal{N})\subseteq \pst_G([q,g^n q])=\pst_G(\M)$. As before,
using minimality of $\pdim_\ga(\pst_G(\L))$, one can conclude that $\pst_G(\mathcal{L})=\pst_G(\cN)\subseteq \pst_G(\M)$.

Let $\mathfrak{M}$ denote the set of axis of the elements of $H$ that act as hyperbolic isometries of $\T$.
We have shown above that \begin{equation}\label{eq:pstL-in_each} \pst_G(\mathcal{L})\subseteq  \pst_G(\M) \mbox{ for every }\M \in \mathfrak{M}. \end{equation}
Note that $H$ naturally acts on the set $\mathfrak{M}$, and so it acts by conjugation on the collection of its subgroups
$\{\pst_G(\M) \mid \M\in \mathfrak{M}\} $
(if $f \in H$ and $\M$ is the axis of a hyperbolic $g \in H$ then $f \pst_G(\M) f^{-1}=\pst_G(f  \M)$, where $f  \M$ is the axis of the
hyperbolic element $fgf^{-1} \in H$). It follows that the intersection $\bigcap_{ \M\in\mathfrak{M}} \pst_G(\M)$ is normalized by $H$; but, according to \eqref{eq:pstL-in_each},
this intersection is equal to $\pst_G(\L)$. Thus the proposition is proved.
\end{proof}

\begin{thm}\label{thm:join}
Suppose that $G$, $H$ and $\T$ are as above. Let $h \in H$ be a hyperbolic element with minimal $\pdim_\ga(\pst_G(\L))$, where $\L$ is the axis of $h$.
If $\pst_G(\L)$ is non-trivial then $H$ is contained in a join subgroup of $G$. And if  $|\pst_G(\L)|=\infty$ then $H$ is contained in a principal join subgroup of $G$.
\end{thm}

\begin{proof} If we set $N:=\pst_G(\L)$, then $N$ is parabolic by Lemma \ref{lem:pst-parab}, thus $N=\pc_\ga(N)$. Clearly the infinite cyclic subgroup $\langle h \rangle$ of $H$
trivially intersects $N:=\pst_G(\L)$, hence $|H:H\cap N|=\infty$. Both claims of the theorem now follow from Lemma \ref{lem:H_in_join}, as
$N$ is normalized by $H$ by Proposition \ref{prop:pst(axis)-normal}.
\end{proof}

\begin{cor}\label{cor:cyc-centr} Let $H$ be a subgroup of a graph product $G=\ga\G$ such that $|\esupp_\ga(H)|\ge 2$ and $H$ is not contained in a join subgroup of $G$.
Then there is a non-trivial element $h \in H$ such that the centralizer $\C_G(h)$ is infinite cyclic.
\end{cor}

\begin{proof} Take any $v \in \esupp_\ga(H)$ and define $A, B, C \subseteq V$, and the Bass-Serre tree $\T$ as above.
Since  $|\esupp_\ga(H)|\ge 2$ and $H$ is not contained in a join subgroup of $G$, we see that
$\esupp_\ga(H) \not\subset \{v\} \cup \link_\ga(v)$ in $\ga$.

According to Lemma \ref{lem:H_has_hyp_elt} and Theorem \ref{thm:join}, there is an element $h \in H$, that acts as a hyperbolic isometry on $\T$, such that $\pst_G(\L)$ is trivial, where $\L$ is the axis of $h$.
Since $h$ has a unique axis (cf. \cite[1.3]{CullerMorgan}), it follows that $\C_G(h) \subseteq \st_G(\L)$. Since $\pst_G(\L)=\{1\}$, $\C_G(h) \leqslant\st_G(\L)$
is isomorphic  to a subgroup of the infinite dihedral group (which corresponds to the group of all simplicial isometries of the line $\L$).
Hence $\C_G(h)$ must be infinite cyclic because it contains the non-trivial central element $h$.
\end{proof}

\subsection{Some consequences of Theorem \ref{thm:join}}\label{sec:conseq}
Again, let $\ga$ be a finite graph with the vertex set $V$ and let $G$ be the graph product of a family of non-trivial groups
$\G=\{G_v \mid v \in V\}$ with respect to $\ga$.
The following result can be regarded as an improvement of claim (ii) in Lemma \ref{lem:parab_clos}.

\begin{thm} \label{thm:one_elt}
Suppose that $H$ is a subgroup of $G$ such that $\esupp_\ga(H)$ is irreducible. Then there exists an element $h \in H$ such that $\pc_\ga(\{h\})=\pc_\ga(H)$.
\end{thm}

Theorem \ref{thm:one_elt} will be derived from the following auxiliary lemma, which is essentially a corollary of Theorem \ref{thm:join}.

\begin{lemma} \label{lem:irred} Assume that $H \leqslant G$, $\pc_\ga(H)=G$, $V$ is irreducible in $\ga$ and $|V|\ge 2$.
Then there exists a non-trivial element $h \in H$ such that $\C_G(h)$ is infinite cyclic and $\pc_\ga(\{h\})=G$.
\end{lemma}

\begin{proof}
Starting with $\ga$ we construct a new graph $\ga'$ as follows. Let $V=\{v_1,\dots,v_k\}$. Take a $k$-element set $U=\{u_1,\dots,u_k\}$ and let the set of vertices
$V'$ of $\ga'$ be defined by $V':=V \sqcup U$. The set of edges $E'$ of $\ga'$ (which we identify with a subset of $V' \times V'$) is defined by
$$E':=E \sqcup \{(v_i,u_j) \mid \mbox{ for all } i \neq j\},$$ where $E$ is the set of edges of $\ga$. Thus $\ga$ is a full subgraph of $\ga'$ and
$|V'|=2k=2|V|$, $|E'|=|E|+k(k-1)$. Basically, $\ga'$ is obtained from $\ga$ by adding a new vertex $u_j$ for each vertex $v_j$ of $\ga$ and connecting $u_j$ by an edge
with every $v_i$ except for $i=j$. For each $j \in \{1,\dots,k\}$ let $G'_{u_j}=\gen{t_j}$ be an infinite cyclic group generated by $t_j$, and set
$\G':=\G \sqcup \{G'_{u_j} \mid j =1,2,\dots,k\}.$ We now have a new graph product $G':=\ga'\G'$ of the family $\G'$ with respect to the finite graph $\ga'$, and
$G=G'_V$ is a full subgroup of $\G'$.

First, note that, by Lemma \ref{lem:par_clos_in_spec},  the parabolic closure $\pc_{\ga'}(H)$ of $H$ in $G'$ is still equal to $G=G'_V$.

Second, let us check that $H$ is not contained in a join subgroup of $G'$. Suppose, on the contrary, that $H \subseteq f G'_T f^{-1}$ for some $T \subseteq V'$ and $f \in G'$,
where $T=A \sqcup B$ such that $A,B$ are disjoint, non-empty and $B \subseteq \link_{\ga'}(A)$. Then $G'_V=\pc_{\ga'}(H) \subseteq fG'_Tf^{-1}$ and so
$V \subseteq T=A\sqcup B$ by Lemma \ref{lem:parab_props}.(i). Note that $V$ is an irreducible subset of $\ga'$ (because
it is irreducible in $\ga$ and $\ga$ is a full subgraph of $\ga'$). Therefore either $V \cap A=\emptyset$ or $V \cap B=\emptyset$. Without loss of generality
assume that $V \cap B=\emptyset$, so that $V \subseteq A$. Hence $B \subseteq U$ and, as $B \neq\emptyset$, there must exist $j \in \{1,\dots,k\}$ with $u_j \in B$.
However, $v_j \in A$ and $u_j \notin \link_{\ga'}(v_j)$ by construction, which contradicts $B \subseteq \link_{\ga'}(A)$.

Thus $H$ is not contained in a join subgroup of $G'$ and all of the assumptions of Corollary~\ref{cor:cyc-centr} are satisfied.
Therefore there exists $h \in H \setminus\{1\}$ such that $\C_{G'}(h)$ is infinite cyclic. Let us prove that $\pc_\ga(\{h\})=G$. Indeed, otherwise
$h \in fG_Sf^{-1}$ for some proper subset $S \subsetneqq V$ and $f \in G$. Then $v_j \notin S$ for some $j \in \{1,\dots,k\}$, hence
$u_j$ is adjacent to every vertex from $S$ in $\ga'$. Consequently, $ft_jf^{-1} \in \C_{G'}(h)$.
Since $t_j$ lies in the kernel of the canonical retraction of $G'$ onto $G=G'_V$, we have $\gen{ft_jf^{-1}}\cap \gen{h} \subseteq \gen{ft_jf^{-1}}\cap G=\{1\}$ in $G'$.
This implies that the infinite cyclic group $\C_{G'}(h)$ contains the
subgroup $\gen{h,ft_jf^{-1}}$, which is naturally isomorphic to $\gen{h} \times \gen{ft_jf^{-1}}\cong \Z^2$. As the latter is impossible,
we conclude that $\pc_\ga(\{h\})=G$, as claimed.
\end{proof}

We are now ready to prove Theorem \ref{thm:one_elt}.

\begin{proof}[Proof of  Theorem \ref{thm:one_elt}] After replacing $H$ with its conjugate, we can assume that $\pc_\ga(H)=G_S$ for some $S \subseteq V$.
As we know, $G_S$ is itself the graph product of the family $\G_S:=\{G_s \mid s \in S\}$ with respect to the full subgraph $\ga_S$ of $\ga$, spanned by the vertices from $S$.
Moreover, $\pc_{\ga_S}(H)=G_S$ since every parabolic subgroup of $G_S$ is also a parabolic subgroup of $G$.

If $|S| \le 1$ then the statement obviously holds, so we can assume that $|S| \ge 2$. Since $S$ is irreducible, by Lemma \ref{lem:irred}
there is $h \in H$ such that the parabolic closure of $\{h\}$ in $G_S=\ga_S\G_S$ is $G_S$.
It remains to apply Lemma \ref{lem:par_clos_in_spec} to conclude that $\pc_\ga(\{h\})=\pc_{\ga_S}(\{h\})=G_S$, as claimed.
\end{proof}

Let us now give an example showing that in Theorem \ref{thm:one_elt} it is indeed necessary to assume that $\esupp_\ga(H)$ is irreducible.
\begin{ex} Let $F$ be a group with an element $f \in F\setminus \{1\}$ of prime order $p\in \N$. Let $G=F^{p+1}$ be the direct product of $p+1$ copies of $F$.
Then $G$ is naturally isomorphic to the graph product of $p+1$ copies of $F$ with respect to the complete graph $\ga$ on $p+1$ vertices.

Set $h_1:=(f,f, \dots,f,1) \in G$, $h_2:=(1,f,f^2,\dots,f^{p-1},f)\in G$ and $H:=\langle h_1,h_2 \rangle \leqslant G$. Clearly $\pc_\ga(H)=G$ and so $\esupp_\ga(H)=V\ga$.
By construction, for any element $h \in H$ there are $m,n \in \Z$ such that $h=h_1^mh_2^n=(f^m,f^{m+n},f^{m+2n},\dots, f^{m+(p-1)n},f^n)$ in $G$.
It is easy to see that if $n \equiv 0 \pmod{p}$ then $f^n=1$ and so $\pc_\ga(\{h\}) \neq G$. On the other hand, if $n \not\equiv 0 \pmod p$ then $m+ln \equiv 0 \pmod p$ for some
$l \in \{0,\dots,p-1\}$, hence $f^{m+ln}=1$ in $G$ and $h$ belongs to a proper parabolic subgroup of $G$.
Thus $\pc_\ga(\{h\}) \neq \pc_\ga(H)$ for all $h \in H$.
\end{ex}

The following is a consequence of Theorem  \ref{thm:join} and
 Corollary \ref{prop:WPD-crit}.

\begin{cor}\label{cor:WPD_in_GP} Suppose that $H$ is a subgroup of $G=\ga\G$ and $v \in \esupp_\ga(H)$ is such that
$\esupp_\ga(H) \not\subset \{v\} \cup \link_\ga(v)$ in $\ga$.
Let $A:=V \setminus \{v\}$, $B:=\{v\} \cup \link_\ga(v)$ and $C:=\link_\ga(v)$. Let $\mathcal T$ be the Bass-Serre tree corresponding to the natural splitting $G=G_A *_{G_C} G_B$.
Assume that $H$ is not contained in a principal join subgroup of $G$. Then there exists $h \in H$  which is a WPD element (for the action of $G$ on $\T$) and
induces a hyperbolic isometry of $\T$. Thus, if $H$ is not virtually cyclic then $H\in \X$.
\end{cor}

\begin{proof} By  Lemma \ref{lem:H_has_hyp_elt}, at least one element of $H$ induces a hyperbolic isometry of $\T$. Hence we can apply Theorem \ref{thm:join} to find
a hyperbolic element $h \in H$ such that $|\pst_G(\L)|<\infty$, where $\L$ denotes the axis of $h$ in $\T$.
Lemma \ref{lem:pst-parab} claims that $\pst_G(\L)=\pst_G([u,v])=\pst_G(\{u,v\})$ for some vertices $u,v \in \L$. Therefore  Corollary \ref{prop:WPD-crit}  applies, telling us
that $h$ is a WPD element for the action of $G$ on $\T$. Clearly the latter yields that $h \in H$ is also a WPD element for the action of $H$ on $\T$.
The last claim of the corollary now follows from Theorem \ref{acyl}.
\end{proof}

If $\ga$ is irreducible, $|V\ga| \ge 2$ and $\esupp_\ga(H)=V$ (equivalently, $\pc_\ga(H)=G$) then any vertex $v \in V$ will satisfy the assumptions of Corollary \ref{cor:WPD_in_GP}. Thus we get the following:

\begin{cor} \label{cor:WPD_in_GP-simplified}
Let $\ga$ be a finite irreducible graph with at least two vertices and let $G =\ga\G$ be the graph product of a family of non-trivial groups $\G=\{G_v \}_{v \in V\ga}$ with respect to $\ga$.
Then there is a simplicial tree $\T$ on which $G$ acts coboundedly and by simplicial isometries such that the following holds. If $H \leqslant G$ is subgroup with $\pc_\ga(H)=G$ then $H$ contains a hyperbolic WPD element
for the action of $G$ on $\T$. Thus either $H$ is virtually cyclic or $H \in \X$.
\end{cor}

\begin{proof}[Proof of Theorem \ref{cor:gp-hyp}]
This theorem is an immediate consequence of Corollary \ref{cor:WPD_in_GP-simplified}.
\end{proof}

\begin{rem} In the case of right angled Artin groups Theorem \ref{cor:gp-hyp} would follow from Theorem 5.2 of \cite{Behr-Char}, because if a subgroup of a CAT($0$)
group contains a rank one element, then it is either virtually cyclic or acylindrically hyperbolic (see \cite[Prop. 3.13, Thm. 5.6]{Sis}).
Unfortunately the proof of this theorem, given in \cite{Behr-Char}, is invalid: the third paragraph of this proof claims that if $c, h$ are elements of a right angled Artin group $G$ and
$c$ is cyclically reduced then in the product $c^k h c^k$ cancellations can only occur between the letters from $c$ and the letters from $h$, which is not always the case
(one can easily construct counterexamples).
Nevertheless, the statement is still true: one can use  Corollary \ref{cor:cyc-centr} above to fix the argument from \cite{Behr-Char}.
\end{rem}

Corollary \ref{thm:gpacyl} is a direct consequence of Theorem \ref{cor:gp-hyp}. As we already mentioned in  Section \ref{sec:results}, it can be considered as another instance of a more general phenomenon:  ``irreducibility" of ``nice" non-positively
curved groups (or spaces) often implies the existence of ``hyperbolic directions" in an appropriate sense. For groups acting on $CAT(0)$ spaces, the
standard formalization of the latter condition is the existence of a \emph{rank one element},
i.e., an axial hyperbolic isometry none of whose axes bounds a flat half-plane.
The notion of a rank one element originates in the celebrated Rank Rigidity
Theorem for Hadamard manifolds proved by Ballmann, Brin, Burns, Eberlein, and Spatzier (for more information and detailed references we refer to \cite{Bal,Bal95,BS87}):
\emph{Let $G$ be a discrete group acting properly and cocompactly on an irreducible Hadamard manifold $M$, which is not a higher rank symmetric space; then $G$ contains a rank one isometry.}
We recall that a Hadamard manifold is a simply connected complete Riemannian manifold of non-positive curvature,
and such a manifold is said to be irreducible if it does not admit a nontrivial decomposition as a metric product of two manifolds.

Motivated by this theorem, Ballman and Buyalo \cite{BB} suggested the following:
\begin{conj}[Rank Rigidity Conjecture]
Let $X$ be a locally compact geodesically complete $CAT(0)$
space and let $G$ be an infinite discrete group acting properly and cocompactly on $X$. If $X$ is
irreducible and is not a higher rank symmetric space or a Euclidean building of dimension at least
$2$, then $G$ contains a rank one isometry.
\end{conj}

Caprace and Sageev \cite{CS} confirmed this conjecture for $CAT(0)$ cube complexes. More precisely, they showed that \emph{if an infinite discrete
group $G$ acts properly and cocompactly on a locally compact geodesically complete $CAT(0)$ cube complex $X$, then $X$ is a product of two geodesically
complete unbounded convex subcomplexes or $G$ contains a rank one isometry.} (Neither higher rank symmetric spaces nor Euclidean buildings of dimension at least $2$ occur in these settings.)

Any rank one element $g$ in a group $G$, acting properly on proper $CAT(0)$ space, is contained in a virtually cyclic subgroup $E(g)$ which is hyperbolically embedded in $G$ \cite{Sis}.
Such subgroups can be thought of as a formalization of ``hyperbolic directions" in the context of abstract groups (for details and the relation
to \emph{generalized loxodromic elements} see \cite[Section 6]{Osi13}). Recall that the existence of a virtually cyclic subgroup $E\h G$ implies that
$G$ is either virtually cyclic (if $G=E$) or acylindrically hyperbolic. Thus the Caprace--Sageev result  \cite{CS}, being applied
to a right angled Artin (respectively, Coxeter) group acting on the universal cover of its Salvetti (respectively, Davis) complex,
yields Corollary \ref{thm:gpacyl}. Note, however, that Theorem \ref{cor:gp-hyp} generalizes this in two directions: it deals with
arbitrary subgroups (which corresponds to non-cocompact actions in the settings of the Caprace-Sageev theorem) and more general
graph products which may not act properly on any $CAT(0)$ spaces. In fact, Theorem \ref{thm:join}, under the assumptions of Theorem \ref{cor:gp-hyp},
produces an action of the subgroup $H$ of a graph product $G$ on a simplicial tree $\T$ so that the pointwise stabilizer of a finite segment of the axis of some hyperbolic element $h \in H$ is finite.
Since a tree cannot contain half-flats, the latter condition can be thought of as an additional requirement one should include in the definition of a rank one isometry for discrete non-proper actions on $CAT(0)$ spaces.

\subsection{Classification of subgroups}\label{sec:subgr-class}
\begin{df}\label{def:N}
Let $\gN$ denote the class of all groups $G$ with the following property: $G$ contains two infinite normal subgroups $N_1,N_2\lhd G$ such that $|N_1\cap N_2|<\infty $.
\end{df}

\begin{lemma}\label{lem: sn} The intersection $\gN\cap \X $ is empty.
\end{lemma}

\begin{proof}
Suppose that $G\in \gN\cap \X$. Thus $G$ contains two infinite normal subgroups $N_1,N_2 \lhd G$ such that $P:=N_1 \cap N_2$ is finite. Let $N:=N_1N_2 \lhd G$,
and observe that $N/P \cong N_1/P \times N_2/P$.
By  Theorem \ref{thm:elem-prop-ah}.(c), we know that $N$ is acylindrically hyperbolic. Therefore, by Lemma \ref{lem: com}, $N_1/P \times N_2/P \in \X$.
However this contradicts Theorem \ref{thm:elem-prop-ah}.(b) because $|N_i/P|=\infty$ for $i=1,2$.
\end{proof}

%Recall that the relation $\fk$, between two groups, denotes commensurability up to finite kernels (see the paragraph above Lemma \ref{lem: com}).
Using Theorem \ref{cor:gp-hyp} we obtain the following classification of subgroups of graph products.

\begin{thm}\label{thm:subgroups} Let $H$ be a subgroup of a graph product $G=\ga\G$ of non-trivial groups with respect to some finite graph $\ga$.
Then at least one of the following holds:
\begin{itemize}
\item[(a)] there is a short exact sequence $\{1\}\to K\to H\to S\to \{1\}$, where $K$ is finite and $S$ is isomorphic to a subgroup of a vertex group;

\item[(b)] $H$ is virtually cyclic;

\item[(c)] $H\in \gN$;

\item[(d)] $H\in \X$.
\end{itemize}
Moreover, the possibilities (b),(c), and (d) are mutually exclusive.
\end{thm}

\begin{proof} %Let $G=\ga\G$ be a graph product of groups for some finite graph $\ga$ and let $H\leqslant G$.
The argument will proceed by induction on $|V\ga|$. If $|V\ga|\le 1$ then (a) holds trivially, so we can assume that $|V\ga |=n\ge 2$
and the statement has already been established for subgroups of graph products with at most $n-1$ vertices.

Without loss of generality we can assume that $H$ is not contained in a proper
parabolic subgroup of $G$, i.e., $\esupp_\ga(H)=V\ga$. Let us first prove that at least one of (a)--(d)
holds.

Suppose that $H$ is not virtually cyclic and $H\notin \X$. Then, by Theorem \ref{cor:gp-hyp}, there exists a decomposition $V\ga=A \sqcup B$,
such that $A,B \neq \emptyset$ and $B \subseteq \link_\ga(A)$. Thus $G$ naturally splits as the direct
product $G=G_A \times G_B$. Let $N_1=H \cap G_A$ and $N_2=H \cap G_B$. Clearly $N_i\lhd H$ for $i=1,2$ and $N_1 \cap N_2=\{1\}$.
If both $N_1$ and $N_2$ are infinite, we have (c). Otherwise $|N_i|<\infty$ for some $i\in \{ 1,2\}$, and then
$H$ splits as $\{1\}\to N_i\to H\stackrel{\rho}{\to} \widehat H\to \{1\}$ where $\widehat H$ a subgroup of $G_A$ or $G_B$. By induction,
$\widehat H$ satisfies one of (a)--(d). If $\widehat H$ satisfies (a) then evidently the same is true for $H$, because the kernel of the epimorphism
$\rho:H \to \widehat{H}$ is finite. By the same reason, $\widehat{H}$ is not virtually cyclic and $\widehat{H}\notin \X$ (by Lemma~\ref{lem: com}).
So, it remains to consider the case when $\widehat{H}$ satisfies (c), i.e., there are infinite normal subgroups ${L},{M} \lhd \widehat{H}$ such
that $|{L} \cap {M}|<\infty$. Then $\rho^{-1}(L),\rho^{-1}(M)$ are infinite normal subgroups of $H$ and
$$|\rho^{-1}(L)\cap \rho^{-1}(M)|=|\rho^{-1}(L \cap M)|=|L \cap M||N_i|<\infty,$$
which shows that $H$ satisfies (c).

It remains to note that (c) and (d) are mutually exclusive by Lemma \ref{lem: sn}, and the fact that (b) is incompatible with (c) and (d) is obvious.
\end{proof}

Corollary \ref{cor:RAAG_in_X} from Section \ref{sec:results} is special case of Theorem \ref{thm:subgroups}, because any right angled Artin group is a
graph product of infinite cyclic groups (we also need to use the well-known fact that right angled Artin groups are torsion-free).

\section{Applications}

\subsection{Some applications to 3-manifold groups}\label{sec:3m-app}
The goal of this section is to use Theorem \ref{main-3d} to obtain some new results about $3$-manifold groups. We begin by proving Corollaries \ref{3-dim-cor-1} and \ref{3-dim-cor-2}.

\begin{proof}[Proof of Corollary \ref{3-dim-cor-1}]
If $G$ is of type (III), it is inner amenable since every amenable group is inner amenable. Further assume that $G$ is of type (II). Since $Z$ is amenable, there exists a finitely additive $Z$-invariant probability measure
$\mu \colon \mathcal P(Z)\to [0,1]$. Given a subset $A\subseteq Z$, we let $ A^-=\{ a^{-1}\mid a\in A\}$. We define a probability measure on $\mathcal P(G\setminus\{1\})$ by the rule
$$\nu (S)=\frac{\mu (S\cap Z) + \mu ((S\cap Z)^-)}2$$ for every $S\subseteq G\setminus \{ 1\}$.
It is easy to see that $\nu $ is finitely additive and invariant under conjugation by elements of $G$ (indeed conjugation by an element of $G$ preserves the collection $\{ A, A^-\}$ for every $A\subseteq Z$).
Thus every group of type (II) is inner amenable.
Finally suppose that $G$ is of type (I), then $G$ has trivial finite radical (see Theorem \ref{main-3d}.(I$^\prime$)) and hence $G$ is not inner amenable by Theorem \ref{thm:prop-ah-2}.
\end{proof}

Using Corollary \ref{cor:geom}, we obtain the following.

\begin{cor}
Let $M$ be a compact irreducible orientable $3$-manifold. Suppose that $\pi_1(M)$ is inner amenable. Then either $M$ is Seifert fibered or $\pi _1(M)$ is amenable.
\end{cor}

Recall that a countable group belongs to the class $w\mathcal C_{reg}$ if its quotient modulo the amenable radical belongs to $\mathcal C_{reg}$ \cite{MS}.

\begin{cor}\label{cor:GANC-3m}
For any $G\in \MM$, the following holds.
\begin{enumerate}
\item[(a)] $G\in \mathcal C_{reg}$ iff $G\in \mathcal D_{reg}$ iff $G$ is of type (I).
\item[(b)] $G\in w\mathcal C_{reg}$ iff $G$ is non-amenable iff $G$ is of type (I) or (II).
\end{enumerate}
\end{cor}

\begin{proof}
Recall that every $G\in \mathcal C_{reg}\cup \mathcal D_{reg}$ is infinite and the amenable radical of any group $G\in \mathcal C_{reg}\cup \mathcal D_{reg}$ is finite \cite{MS,T}. Thus $C\notin \mathcal C_{reg}\cup \mathcal D_{reg}$ whenever $C$ is of type (II) or (III). On the other hand, $G\in  \mathcal C_{reg}\cap \mathcal D_{reg}$ whenever $G$ is of type (I) by Theorem \ref{thm:prop-ah-3}. This proves (a); (b) obviously follows from (a).
\end{proof}

\begin{proof}[Proof of Corollary \ref{3-dim-cor-2}]
It is proved in \cite[Cor. 7.6 and 7.9]{MS} that being in $\mathcal C_{reg}$ and $w\mathcal C_{reg}$ is a measure equivalence invariant.
On the other hand, it is well known that being amenable is also a measure equivalence invariant (see, for example, \cite{Fur}). Thus the claim follows from Corollary \ref{cor:GANC-3m}.
\end{proof}

We mention one more application. For the definition and details about conjugacy growth we refer to \cite{HO13}. Recall that, in general, any non-decreasing function $\mathbb N\to \mathbb N$ bounded from above by an exponential function can be realized up to a standard equivalence as the conjugacy growth function of a finitely generated group \cite{HO13}. For $3$-manifolds groups, the situation is more rigid.

\begin{cor}
Let $M$ be a compact orientable $3$-manifold. Then $\pi_1(M)$ has either exponential or polynomially bounded conjugacy growth. Moreover, the conjugacy growth function of $\pi_1(M)$ is polynomially bounded if and only if $\pi_1(M)$ is virtually nilpotent.
\end{cor}

\begin{proof}
By Theorem  \ref{thm:prop-ah-4}, the conjugacy growth of $\pi_1(M)$ is exponential if $\pi_1(M)$ has type (I) or (II). On the other hand, it is proved in \cite{Hull11} that if $\pi_1(M)$ is of type (III), its conjugacy growth is either exponential or polynomially bounded and the latter possibility happens iff $\pi_1(M)$ is virtually nilpotent.
\end{proof}

It is worth noting that, in some cases, one can obtain results similar to the above corollaries using relative hyperbolicity. For example, it is not hard to prove that if $M$ is a closed orientable $3$-manifold with a hyperbolic piece in the JSJ decomposition, then $\pi_1(M)$ is hyperbolic relative to proper subgroups. However this approach does not always work (for example for graph manifolds).

\subsection{Geometric vs. analytic negative curvature} \label{sec:nc}
Let $\mathcal {SN}$  be the class of all groups $G$ containing a subnormal subgroup $N$ and a finite subgroup $K\lhd N$ such that $N/K$ splits as a direct product two infinite groups.
Let also $\mathcal N$ be the class defined in Definition \ref{def:N}.

\begin{lemma}
$\mathcal N\subset \mathcal {SN}$.
\end{lemma}

\begin{proof}
If $G$ contains two infinite normal subgroups $N_1,N_2\lhd G$ such that $|N_1\cap N_2|<\infty $, then $N=N_1N_2$ is a normal subgroup of $G$ and $N/(N_1\cap N_2)$ decomposes as a direct product of two infinite groups (isomorphic to $N_1/(N_1\cap N_2)$ and $N_2/(N_1\cap N_2)$).
\end{proof}

The next observation is an easy consequence of known results.
\begin{lemma}\label{CDN}
\begin{enumerate}
\item[(a)] $\mathcal{C}_{reg}\cap \mathcal {SN}=\emptyset$.
\item[(b)] $\mathcal{D}_{reg}\cap \mathcal {SN}=\emptyset$.
\end{enumerate}
\end{lemma}
\begin{proof}
Let $G \in  \mathcal {SN} $ be a group containing a subnormal subgroup $N$ and a finite subgroup $K\lhd N$ such that $N/K$ splits as a direct product two infinite groups.

Suppose first that $G\in \mathcal{C}_{reg}$. The class $\mathcal{C}_{reg}$ is closed under taking infinite subnormal subgroups by \cite[Prop. 7.4]{MS}. Hence $N\in \mathcal{C}_{reg}$.  Further we have $N/K\in \mathcal{C}_{reg} $ by \cite[Lemma~7.3]{MS}.
However this contradicts \cite[Prop. 7.10.(iii)]{MS}, which states that no group from $\mathcal{C}_{reg}$  decomposes as a direct product of two infinite groups. Therefore  $\mathcal{C}_{reg}\cap \mathcal {SN}=\emptyset$.

Suppose now that $G\in \mathcal{D}_{reg}$. By \cite[Thm 3.4]{T}, every infinite normal subgroup of a group from $\mathcal{D}_{reg}$ itself belongs to $\mathcal{D}_{reg}$. Hence we obtain $N\in \mathcal{D}_{reg}$ as above. Further by \cite[Lemma 2.8]{T}, $N$ either belongs to $\mathcal{C}_{reg}$ or has
non-vanishing first $\ell^2$-Betti number. We have already shown that $N\in \mathcal{C}_{reg}$ is impossible. If $\beta_1^{(2)} (N) >0$, then we have
$\beta_1^{(2)} (N/K) >0$, say by \cite[Thm. 7.34]{L}. However the latter inequality is impossible since $N/K$ decomposes as a direct
product of two infinite groups (see \cite[Thm.~6.54.(5)]{L}).
\end{proof}

\begin{proof}[Proof of Corollary \ref{cor:GANC}]
Claim (a) follows from Corollary \ref{cor:GANC-3m}.(a), so it remains to prove claim (b).

Let $G$ be a graph product of amenable groups, let $H \leqslant G$.
If $H$ is amenable, then it does not belong to either of the classes $\X$, $\mathcal C_{reg}$, $\mathcal{D}_{reg}$.
Thus we can assume that $H$ is non-amenable. Then, according to Theorem \ref{thm:subgroups}, either $H \in \mathcal{N}$ or $H \in \X$.

Assume, first, that $H \in \mathcal{N}$. Then $H\in \mathcal {SN}$, which implies that $H\notin \mathcal{C}_{reg}\cup \mathcal{D}_{reg}$ by Lemma \ref{CDN}. 
We also have $H\notin \X$ by Lemma~\ref{lem: sn}.

Finally,  if $H \in \X$, then $H\in   \mathcal C_{reg}\cap \mathcal D_{reg}$ as $\X \subseteq \mathcal C_{reg}\cap \mathcal D_{reg}$ by Theorem \ref{thm:prop-ah-3}.
\end{proof}

\subsection{Local properties of acylindrically hyperbolic groups}
As demonstrated by Theorems \ref{thm:elem-prop-ah}--\ref{thm:prop-ah-4}, global properties of acylindrically hyperbolic groups resemble those of relatively hyperbolic groups.
One can then ask if the same is true for \emph{local} properties (i.e., properties that pass to finitely generated subgroups).
If a group $G$ is hyperbolic relative to a subgroup $H$ then many  local properties of $H$ are inherited by $G$ (e.g., if $H$ is finitely generated and has solvable word problem then
the same is true for $G$ -- see \cite{Farb}, if $H$ has finite asymptotic dimension then so does $G$ -- see \cite{O-asdim}, etc.). However, the same cannot be said if $H$ is hyperbolically embedded in $G$.
Indeed, if a property $P$ is {local}, one can simply take some finitely generated group
$A$ which does not possess this property and set $G=H*A$. Then $H$ will be hyperbolically embedded into $G$ but $G$ will not satisfy $P$ even if $H$ does satisfy it. Of course
in this example both $H$ and $A$ are hyperbolically embedded into $G$ (and $G$ is hyperbolic relative to $\{H,A\}$),
and so, naturally, to ensure ``niceness'' of $G$ one should require that both $H$ and $A$ are ``nice''.
This leads to the following question:

\begin{prob} Let $G$ be a finitely generated acylindrically hyperbolic group. Suppose that every proper hyperbolically embedded subgroup $H \h G$ satisfies some local property $P$.
Does $G$ also have this property?
\end{prob}

Our investigation of graph products allows to produce simple examples showing that the answer to the above question is almost always negative. Namely, we have the following statement:

\begin{thm}\label{thm:emb} Any finitely generated group $K$ can be embedded into a finitely generated group $G$ so that all of the following hold:
\begin{enumerate}
  \item[(a)] $G \in \X$;
  \item[(b)] every proper hyperbolically embedded subgroup $H \h G$ is finitely generated and virtually free;
  \item[(c)] $K$ is a retract of $G$;
  \item[(d)] if $K$ is torsion-free then so is $G$;
 \item[(e)] if $K$ is finitely presented then so is $G$.
\end{enumerate}
\end{thm}

\begin{proof} Let $\ga$ be the simplicial path of length $3$ with (consecutive) vertices $a,b,c,d$, where $a$ and $d$ have valency $1$ and $b,c$ have valency $2$.
Let $G_a:=K$ and let $G_b$, $G_c$ and $G_d$ be infinite cyclic groups generated by elements $x_b,x_c$ and $x_d$ respectively. Let $G=\ga\G$ be the graph product of these groups
with respect to $\ga$ (see Figure \ref{fig:4}).
Clearly $\ga$ is irreducible and $G$ is not virtually cyclic, hence $G \in \X$ by Theorem \ref{cor:gp-hyp}. The properties (c)-(e) of $G$ follow from general properties of graph products, so it remains only to verify property (b).

\begin{figure}[ht]
\input{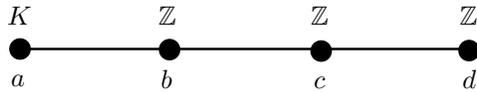}
\caption{The graph $\Gamma$ and the graph product $G=\ga\G$.}\label{fig:4}
\end{figure}

By the definition, $G$ splits as the amalgamated free product $G=G_{ab}*_{G_b} G_{bcd}$,
where $G_{ab}:=K \times \langle x_b \rangle$, $G_b=\langle x_b \rangle$ and $G_{bcd}:=\langle x_b,x_c,x_d \rangle$.
Then $G$ acts on the Bass-Serre tree $\T$ corresponding to this splitting with vertex stabilizers being conjugates of $G_{ab}$ or $G_{bcd}$.
Note that $\langle x_b \rangle$ is central in $G_{ab}$ and $\langle x_c \rangle$ is central in $G_{bcd}$.

Consider any $H \h G$. Then $H$ is \emph{almost malnormal} in $G$, i.e., $|H \cap H^f|<\infty$ for any $f\in G \setminus H$ (see \cite[Prop. 4.33]{DGO}).
Suppose that the intersection $H \cap G_{ab}^g$ is infinite for some $g \in G$.
Since $H_1:=g^{-1}Hg$ is  the image of $H$ under an inner automorphism, it is also almost malnormal in $G$.

Since $x_b$ is central in $G_{ab}$, $H_1 \cap G_{ab} \subseteq H_1 \cap H_1^{x_b}$ is infinite, hence $x_b\in H_1$ by almost malnormality of $H_1$.
But then for any $f \in G_{ab}$, the infinite order element $x_b$ belongs to $H_1 \cap H_1^f$, implying that $f \in H_1$. Thus $G_{ab} \subseteq H_1$.
Now, since $x_c$ commutes with $x_b \in H_1$, we can similarly get that $x_c \in H_1$, and since $x_d$ commutes with $x_c$ we conclude that
$x_d \in H_1$. Thus all the generators of $G$ are contained in $H_1$, therefore $H_1=G$, yielding that $H=G$. Similarly, one can show that $H=G$ if
$H$ has infinite intersection with some conjugate of $G_{bcd}$ (using the fact that $x_c$ is central in $G_{bcd}$).

Thus if $H \h G$ is a proper subgroup, then $H$ has finite intersection with every vertex stabilizer for the action of $G$ on $\T$. Observe that $G$ is finitely generated because $K$ is,
hence we can apply \cite[Cor. 4.32]{DGO} claiming that $H$ is finitely generated as well. By \cite[I.4.12]{D-D} $H$ either fixes a vertex of $\T$ or contains an element that acts on $\T$
as a hyperbolic isometry. In the former case $H$ will have to be finite (because it acts with finite vertex stabilizers), and in the latter case by \cite[I.4.13]{D-D} there is a unique minimal
$H$-invariant subtree $\T'$ of $\T$ on which $H$ acts cocompactly. Now we can apply
\cite[IV.1.9]{D-D} to conclude that $H$ is virtually free, as desired.
\end{proof}

For example, if $K$ is a finitely presented torsion-free group with unsolvable word problem, then, by Theorem \ref{thm:emb}, $K$ can be embedded into a finitely presented
torsion-free acylindrically hyperbolic group $G$ such that every hyperbolically embedded subgroup in $G$ is free of finite rank.

Yet another application is related to maximal elliptic subgroups (a subgroup is \emph{elliptic} if all of its orbits are bounded).
Recall that for every group $G$ acting acylindrically on a hyperbolic space,
every elliptic subgroup of $G$ is contained in a maximal elliptic subgroup of $G$, and maximal elliptic subgroups exhibit malnormality-type properties similar to
hyperbolically embedded subgroups (see \cite[Corollary 1.6]{Osi13}). It is easy to show that if $H\h (G,X)$ and $H$ is infinite, then $H$ is maximal elliptic
with respect to the action of $G$ on the Cayley graph $\Gamma (G, X\sqcup H)$. Thus it is natural to ask whether maximal elliptic subgroups are hyperbolically embedded.
The following corollary provides the negative answer. In what follows we say that a subgroup $H\leqslant G$ is \emph{absolutely elliptic} if it is elliptic with
respect to any acylindrical action of $G$ on a hyperbolic space.

\begin{cor}
There exists a group $G\in \X$ and an absolutely elliptic subgroup $K\leqslant G$ such that $K$ is not contained in any
proper hyperbolically embedded subgroup of $G$. In particular, for any non-elliptic acylindrical action of $G$ on a hyperbolic space, the maximal elliptic subgroup of $G$ containing $K$ is not hyperbolically embedded.
\end{cor}

\begin{proof}
Let $K= \mathbb Z\oplus \mathbb Z$ and let $G$ be the group provided by Theorem \ref{thm:emb}. Note that $K$ is
absolutely elliptic by Theorem \ref{class} (it is not acylindrically hyperbolic by Theorem \ref{thm:elem-prop-ah}).
On the other hand, it cannot be a subgroup of any proper hyperbolically embedded subgroup of $G$ by part (b) of Theorem \ref{thm:emb}.
\end{proof}

\section{Open problems}\label{sec:q}
We showed that every one-relator group with at least $3$ generators and ``most" one-relator group with least $2$ generators  are acylindrically hyperbolic. However the case of $2$-generator groups is far from being clear.

\begin{prob}\label{p1}
Which one-relator $2$-generated groups are acylindrically hyperbolic?
\end{prob}

By Proposition \ref{cor:1-rel-2-gen}, to solve Problem \ref{p1} it suffices to consider one-relator groups which are ascending HNN-extensions of free groups. This leads to the following:

\begin{prob}\label{p2}
Which mapping tori of injective endomorphisms of finitely generated free groups are acylindrically hyperbolic?
\end{prob}

In the other direction, one can try to generalize our Corollary \ref{cor:one-rel} as follows:

\begin{prob}\label{p3}
Is every group of deficiency at least $2$ acylindrically hyperbolic?
\end{prob}

As a warm up for Problem \ref{p3}, we mention the following question, which looks much easier. We expect the answer to be negative, but we are not aware of any counterexamples.

\begin{prob}
Is every group of deficiency at least $2$ hyperbolic relative to a collection of proper subgroups?
\end{prob}

The next two questions are motivated by our results about the integral $2$-dimensional Cremona group ${\rm Aut}\, k[x,y]$.
Wright \cite{Wri} proved that the $2$-dimensional Cremona group $Cr_2$ (i.e., the group of birational transformations of the complex projective plane) is isomorphic to the
fundamental group of a finite developable complex of groups. Let $\mathcal W$ denote its development. Thus $Cr_2$ acts on $\mathcal W$ cocompactly.
The positive answer to the following question could lead to a new way of proving that
$Cr_2$ is acylindrically hyperbolic and hence not simple. Note that acylindrical hyperbolicity of $Cr_2$ was proved in \cite{DGO} and non-simplicity of
$Cr_2$ was a long-standing conjecture in algebraic geometry resolved in \cite{CL}.

\begin{prob}
Is $\mathcal W$ hyperbolic (as a metric space)?
\end{prob}

The last problem is inspired by Corollary \ref{GA2}. Note that for $n\ge 3$ no nontrivial decomposition of ${\rm Aut}\,k[x_1,\ldots, x_n]$  in an amalgamated product is known,
so our strategy of the proof of Corollary \ref{GA2} does not work. In fact, we rather expect the answer to be negative.

\begin{prob}
Let $k$ be a field. Is ${\rm Aut}\,k[x_1,\ldots, x_n]$ acylindrically hyperbolic for $n\ge 3 $?
\end{prob}

Finally, we would like to mention that in the originally published version of the article the statement of Lemma \ref{lem: com} was stronger: it claimed that acylindrical hyperbolicity is invariant under commensurability up to finite kernels. Unfortunately the proof contained a mistake (see \cite{erratum}), and we do not know the answer to the following question.

\begin{prob} Suppose that a group $G$ contains a finite index subgroup which is acylindrically hyperbolic. Does it follow that $G$ is acylindrically hyperbolic itself?
\end{prob}

\end{document}